\documentclass[11pt,twoside]{amsart}
\usepackage{amssymb,latexsym,amsthm,amsmath,mathtools,amsfonts,hyperref,fancyhdr,comment, tabularht, tabularx, longtable, mathrsfs, cleveref,booktabs, color, geometry,bm}
\usepackage[
singlelinecheck=false 
]{caption}

\usepackage{geometry}
\hypersetup{
	colorlinks=true,
	linkcolor=blue,
	filecolor=magenta,
	urlcolor=cyan,
	citecolor=red,
}

\long\def\symbolfootnote[#1]#2{\begingroup%
	\def\thefootnote{\fnsymbol{footnote}}\footnote[#1]{#2}\endgroup}

\newcommand{\C}{\ensuremath{\mathscr{C}}}
\newcommand{\D}{\ensuremath{\mathscr{D}}}
\newcommand{\Z}{\ensuremath{\mathcal{Z}}}

\newcommand{\tr}{\mathsf{tr}}

\newcommand{\n}{\mathfrak n}
\newcommand{\x}{\mathfrak X}
\newcommand{\h}{\mathfrak h}

\newcommand{\SL}{\textup{SL}}
\newcommand{\GL}{\textup{GL}}
\newcommand{\psl}{\textup{PSL}}

\newcommand{\F}{\mathbb{F}}

\makeatletter
\def\imod#1{\allowbreak\mkern10mu({\operator@font mod}\,\,#1)}
\makeatother

\newtheorem{theorem}{Theorem}[section]
\newtheorem{lemma}[theorem]{Lemma}
\newtheorem{corollary}[theorem]{Corollary}
\newtheorem{proposition}[theorem]{Proposition}
\newtheorem*{theorem*}{Theorem}
\theoremstyle{definition}
\newtheorem{definition}[theorem]{Definition}
\newtheorem{remark}[theorem]{Remark}
\newtheorem{obs}[theorem]{Observation}

\numberwithin{equation}{section}
\newcommand{\ignore}[1]{}

\newcommand{\mynote}[1]{}

\usepackage[labelsep=period]{caption}
\begin{document}
	
	\title{Products of conjugacy classes in $\SL_2(k)$ and $\psl_2(k)$}
	
	\author{Harish Kishnani}
	\address{Indian Institute of Science Education and Research Mohali, Knowledge City, Sector 81, Mohali 140306, India}
	\email{harishkishnani11@gmail.com}
	
	\author{Rijubrata Kundu}
	\address{Indian Institute of Science Education and Research Mohali, Knowledge City, Sector 81, Mohali 140306, India}
	\email{rijubrata8@gmail.com}
	
	\author{Sumit Chandra Mishra}
	\address{Indian Institute of Science Education and Research Mohali, Knowledge City, Sector 81, Mohali 140306, India}
	\email{sumitcmishra@gmail.com}
	
	\today 
	\begin{abstract}
		Let $k$ be a field with $u$-invariant $\leq2$. Assume further that $k$ is not quadratically closed, $\mathsf{char}(k)\neq 2$ and $|k|\geq 5$. It is known that the covering number of both $\SL_2(k)$ and $\psl_2(k)$ is three, while their  extended covering number is four. In this article, we completely describe the product of two conjugacy classes in $\SL_2(k)$ and $\psl_2(k)$. Further, we also describe the product of three conjugacy classes (at least two of which are distinct) in $\SL_2(k)$ and $\psl_2(k)$.
	\end{abstract}
	
	\maketitle
	
	\section{Introduction}
	Let $G$ be a group. The covering number of $G$ (if it exists), denoted by $cn(G)$, is the least natural number $n$ such that $C^n=G$ for every non-central conjugacy class $C$ of $G$. More generally, the extended covering number denoted by $ecn(G)$ (if it exists) is the least natural number $n$ such that for any collection of non-central conjugacy classes $C_1,\ldots,C_n$, we have $C_1C_2\cdots C_n=G$.  For a quasisimple Chevalley group $G$, Ellers, Gordeev, and Herzog (see \cite{egh}) obtained an upper bound for $cn(G)$ which is linear in terms of the rank of $G$. Gordeev and Saxl obtained the same for $ecn(G)$ (see \cite{gs}). The covering number and the extended covering number of the alternating groups and the sporadic groups are known (see \cite{ash} for more details).
	
 Lev (see \cite{le,le1} and the references therein) has studied covering properties of conjugacy classes in the simple groups $\psl_n(k)$, where $k$ is a field with $|k|\geq 4$. The author proved that $cn(\psl_n(k))=n$ if $n\geq 4$, or $n=3$ and $k$ is either finite or algebraically closed field (see  \cite[Theorem 1]{le}). Further, in \cite{kn}, Kn\"uppel and Nielsen showed that $ecn(\SL_n(k))=n+1$ if $n\geq 3$, $n\neq 4$ and $|k|\geq 3$. On the contrary, the values of $cn(\SL_2(k))$, $cn(\psl_2(k))$, $ecn(\SL_2(k))$, and $ecn(\psl_2(k))$ depend on the properties of the underlying field $k$ (see \cite{vw}).
	
	A field $k$ is called quadratically closed if it does not admit any quadratic extension. A field $k$ is said to be of $u$-invariant $\leq n$ if every quadratic form in $n+1$ variables over $k$ is isotropic. In this case, we also write $u(k)\leq n$. Vaserstein and Wheland (see \cite{vw}) proved that $cn(\psl_2(k))=2$ and $ecn(\psl_2(k))=3$ when $k$ is quadratically closed field. When $k$ is of $u$-invariant $\leq 2$, not quadratically closed, $\mathsf{char}(k)\neq 2$ and $|k|\geq 5$, they proved that $cn(\psl_2(k))=3$ and $ecn(\psl_2(k))=4.$ For more general fields, the computations are more complicated and some partial results are given in the same article.

	In this article, we give complete description of the product of two conjugacy classes in $\SL_2(k)$ and $\psl_2(k)$, where $k$ is a field with $u(k)\leq 2$, $k$ is not quadratically closed, $\mathsf{char}(k)\neq 2$, and $|k|\geq 5$; though for some results for $\SL_2(k)$, we need to assume that $|k|>5$ (see \Cref{unipotent_square_description} and \Cref{negative_unipotent_square_description}). We also describe the product of three conjugacy classes (with at least two of them being distinct) for the same. Some examples of fields which are of $u$-invariant $\leq 2$ and not quadratically closed are finite fields and function fields of algebraic curves over algebraically closed fields. Garion has described the square of conjugacy classes in $\psl_2(q)$ (see \cite[Theorem 2 \& Theorem 3]{ga}). Thus our results generalize these over a larger class of fields with characteristic different from $2$.

	Before stating the main results, we set down some notations and conventions. Throughout this article (except for Section \ref{commutator_map}), unless mentioned otherwise, we consider fields $k$ which are not quadratically closed with $\mathsf{char}(k)\neq 2$ and $|k|\geq 5$; although some of the results hold over more general fields as well. For a field $k$, we denote the set of squares in $k^{\times}$ by $k^{\times 2}$. For $\epsilon \in k^{\times}$, $\bar{\epsilon}$ denotes its image in $k^{\times}/k^{\times 2}$. For a matrix $x$, $\tr(x)$ denotes its trace.  For $x\in \SL_2(k)$, $\bar{x}$ denotes its image in $\psl_2(k)$. By a semisimple element (resp. semisimple conjugacy class) of $\SL_2(k)$ or $\psl_2(k)$, we mean a non-central semisimple element (resp. non-central semisimple conjugacy class). For a group $G$ and $x\in G$, $|x|$ denotes the order of $x$. Now we can state our main results.
	
	\begin{theorem}\label{square_semisimple_class_PSL}
		Let $k$ be a field with $u(k)\leq 2$ and let $G:=\psl_2(k)$. Let $\C$ be a semisimple conjugacy class in $G$. 
		\begin{enumerate}
			\item If $\C$ is split semisimple then $\C^2=G$
			\item If $\C$ is non-split semisimple then $\C^2=G$ unless $tr(\C)=0$, in which case $-1\notin k^{\times 2}$ and 
			$$\C^2=G\setminus \{unipotents\}.$$
		\end{enumerate}
		In particular, if $-1\in k^{\times 2}$, then $\C^2=G$.
	\end{theorem}
	
	\begin{theorem}\label{product_distinct_semisimple_class_PSL}
		Let $k$ be a field with $u(k)\leq 2$ and let $G:=\psl_2(k)$. Let $\C_1,\C_2$ be two distinct semisimple conjugacy classes in $G$. Then $\C_1\C_2=G\setminus\{1\}$.
	\end{theorem}
	
	\noindent Note that \Cref{square_semisimple_class_PSL} is a generalization of \cite[Theorem 3(i)-(ii)]{ga} over fields of $u$-invariant $\leq 2$. We make use of the Bruhat decomposition in $\SL_2(k)$ as one of our key tools in proving the above theorems.
	
	\begin{theorem}\label{product_of_distinct_unipotent_class_PSL}
		Let $k$ be a field with $u(k)\leq 2$ and let $G:=\psl_2(k)$. Denote by $\bar{x}$, the image of an element $x\;(\in \SL_2(k)$) in $G$. Suppose that $\overline{\epsilon_1},\overline{\epsilon_2}\in k^{\times}/k^{\times 2}$. Let $\C_1$ and $\C_2$ be the  unipotent conjugacy classes of  $\overline{\begin{pmatrix}1 & \epsilon_1\\ 0 & 1\end{pmatrix}}$ and  $\overline{\begin{pmatrix}1 & \epsilon_2\\ 0 & 1\end{pmatrix}}$ in $G$ respectively. Let $\Phi$ be the set of semisimple elements $\bar{x}\in G$ such that one of $2-\tr(x)$ or $2+\tr(x)$ is in $\epsilon_1\epsilon_2k^{\times 2}$. Then 
		\[\C_1\C_2=
		\begin{cases}
			\{unipotents\}\cup\Phi \cup \{1\} & \text{if } \overline{\epsilon_2}=\overline{-\epsilon_1},\\
			\{unipotents\}\cup\Phi & \text{otherwise}.
		\end{cases}
		\]
	\end{theorem}
	
	Taking $\overline{\epsilon_1}=\overline{\epsilon_2}$ in the above theorem, we get the description of the square of a unipotent conjugacy class in $\psl_2(k)$. In particular, in that case, $\Phi$ is the set of semisimple elements $\bar{x}\in \psl_2(k)$ such that one of $2-\tr(x)$ or $2+\tr(x)$ is a square in $k^{\times}$. Also note that this case is a  generalization of \cite[Theorem 3(iii)]{ga} over fields with $u$-invariant $\leq 2$, although Garion's description of the semisimple elements contained in the square of a unipotent conjugacy class in $\psl_2(q)$ is in terms of the orders of such elements. Thus, we restate \Cref{product_of_distinct_unipotent_class_PSL}  for the product of two distinct unipotent conjugacy classes in $\psl_2(q)$, in terms of orders of its elements (see \Cref{product_distinct_unipotent_finite_fields}).
	
	\medskip
	
	We need the following definition for stating the next theorem. Let $\bar{x}\in \psl_2(k)$ be a semisimple element, where $x\in \SL_2(k)$. Suppose that $\overline{\epsilon}\in k^{\times}/k^{\times 2}$.  Let $$\Delta_{\overline{\epsilon},\bar{x}}:=\{\overline{\epsilon_1}\in k^{\times}/k^{\times 2}\mid one \;of\;\tr(x)-2\; or\; -\tr(x)-2\;belongs\;to\; \epsilon\epsilon_1k^{\times 2}\}.$$ Let $y\in \SL_2(k)$. We say that $\bar{y}\in \psl_2(k)$ is a \textbf{\boldsymbol{$(\overline{\epsilon},\bar{x})$}-admissible} unipotent element if $\bar{y}$ is conjugate to $\overline{\begin{pmatrix}1 & \epsilon_1\\0 & 1\end{pmatrix}}$ in $\psl_2(k)$ for some $\overline{\epsilon_1}\in \Delta_{\overline{\epsilon},\bar{x}}$.
	
	\begin{theorem}\label{semisimple_times_unipotent_PSL}
		Let $k$ be a field with $u(k)\leq 2$ and let $G:=\psl_2(k)$. Let $\C_1$ be the semisimple conjugacy class of $\bar{x}\in G$, where $x\in \SL_2(k)$. Suppose that $\overline{\epsilon}\in k^{\times}/k^{\times 2}$. Let $\C_2$ be the unipotent conjugacy class of  $\overline{\begin{pmatrix}1 & \epsilon\\ 0 & 1\end{pmatrix}}$ in $G$. 
		\begin{enumerate}
			\item If $\tr(x)\neq 0$ or $-1\in k^{\times 2}$, then
			\[\C_1\C_2=\{semisimple\;elements\}\cup \{\ (\overline{\epsilon},\bar{x})\text{-}admissible\; unipotent\;elements\}.\]
			
			\item If $\tr(x)=0$ and $-1\notin k^{\times 2}$ (whence $\C_1$ is a non-split semisimple class), then
			\[\C_1\C_2=\{semisimple\;elements\;not\;in\;\C_1\}\cup \{(\overline{\epsilon},\bar{x})\text{-}admissible\; unipotent\;elements\}.\]
		\end{enumerate}
	\end{theorem}
	
	All of the above theorems are proved in \Cref{product_of_two_classes}. The complete description of the product of two conjugacy classes in $\SL_2(k)$ can be found in \Cref{square_semisimple_class_SL}, \Cref{product_distinct_semisimple_class_SL}, \Cref{unipotent_square_description}, \Cref{product_of_two_distinct_unipotents}, \Cref{negative_unipotent_square_description}, \Cref{product_of_two_distinct_negative_unipotents}, \Cref{product_of_unipotent_negative_unipotent_class_SL}, \Cref{semisimple_times_unipotent} and \Cref{semisimple_times_negative_unipotent}. 
	
	It is worthwhile to mention the Arad-Herzog conjecture, which states that for a non-abelian finite simple group, the product of two non-central conjugacy classes is never a single conjugacy class. This has been proved in many cases; Guralnick, Malle and Tiep (see \cite[Theorem 2.5]{gmt}) proved the conjecture for the groups $\psl_n(q)$. In the same paper (see \cite[Theorem 1.1]{gmt}), it was also proved that the natural analogue of this conjecture holds for simple groups over algebraically closed fields; actually, it was shown that the product contains infinitely many conjugacy classes except for few kinds of groups. Our results imply that the analogue of Arad-Herzog conjecture holds for $\psl_2(k)$ over fields $k$ with $u(k)\leq 2$, $\mathsf{char}(k)\neq 2$ and $|k| \geq 5$ such that $k$ is not quadratically closed. Furthermore, if we also assume that $k$ is infinite, then the product contains infinitely many conjugacy classes.

	We state our final result concerning the product of three conjugacy classes in $\psl_2(k)$, at least two of which are distinct.
	
	\begin{theorem}\label{product_of_three_classes_PSL}
		Let $k$ be a field with $u(k)\leq 2$ and let $G:=\psl_2(k)$. Let $\C_1,\C_2,\C_3$ be three conjugacy classes in $G$, at least two of which are distinct. Then $G\setminus\{1\}\subseteq\C_1\C_2\C_3$.
	\end{theorem}

	\noindent We prove this result in \Cref{product_of_three_classes}. We remark that using \Cref{square_semisimple_class_PSL}-\Cref{semisimple_times_unipotent_PSL}, it is easy to determine when $1 \in \C_1\C_2\C_3$ in the above theorem. We omit the details here. We conclude our article with \Cref{commutator_map}, where we apply our results to determine when an element of $\psl_2(k)$ (or $\SL_2(k)$) can be written as a commutator of a semisimple element and a unipotent element. This discussion relates to a recent conjecture on \textbf{coprime commutators} in finite non-abelian simple groups by P. Shumyatsky (see \cite[Section 3]{sh}).  
	
	\section{Preliminaries}\label{preliminaries}
	
	In this section, we collect some well-known facts on conjugacy classes in $\SL_2(k)$ and $\psl_2(k)$. 
	
	\subsection{Conjugacy classes in $\SL_2(k)$ where $k$ is a field with $u(k)\leq 2$} Let  $A \in \GL_2(k)$ be a non-central element. The characteristic polynomial of $A$ (which is a monic quadratic polynomial over $k$ with non-zero constant term) determines $A$ upto conjugacy, that is,
	
	\begin{lemma}\label{conjugacy_GL}
		Let $A,B$ be two non-scalar matrices in $\GL_2(k)$. Then $A$ is conjugate to $B$ in $\GL_2(k)$ if and only if their characteristic polynomials coincide.
	\end{lemma}
	
	Suppose $g\in \SL_2(k)$. Then the characteristic polynomial of $g$ is given by $x^2-\lambda x+1$, where $\lambda$ is the trace of $g$. We say that an element $g\in \GL_2(k)$ is  \textbf{negative unipotent} if its minimal polynomial is $(x+1)^2$. The following lemma summarizes the conjugacy classes in $\SL_2(k)$. 
	
	\begin{proposition}\label{conjugacy_SL}
		Let $k$ be a field with $u(k)\leq 2$ and let $x,y$ be two non-central elements of $\SL_2(k)$. 
		\begin{enumerate}
			\item If $x$ and $y$  are semisimple, then $x$ and $y$ are conjugate in $\GL_2(k)$ if and only if $x$ and $y$ are conjugate in $\SL_2(k)$. In particular, $x$ and $y$ are conjugate in $\SL_2(k)$ if and only if $\tr(x)=\tr(y)$.
			
			\item The unipotent conjugacy classes in $\SL_2(k)$ are in one-one correspondence with elements of $k^{\times}/k^{\times 2}$. More precisely, for $\epsilon_1,\epsilon_2\in k^{\times}$,  the $\SL_2(k)$-conjugacy class of  $\begin{pmatrix}1 & \epsilon_1 \\ 0 & 1\end{pmatrix}$ and the $\SL_2(k)$-conjugacy class of $\begin{pmatrix} 1 & \epsilon_2\\ 0 & 1\end{pmatrix}$ coincide if and only if $\overline{\epsilon_1}=\overline{\epsilon_2}$ in $k^{\times}/k^{\times 2}$. 
			
			\item For $\overline{\epsilon}\in k^{\times}/k^{\times 2}$, let $U_{\overline{\epsilon}}$ denote the $\SL_2(k)$-conjugacy class of  $\begin{pmatrix}1 & \epsilon\\ 0 & 1\end{pmatrix}$. Then 
			$$U_{\overline{\epsilon}}= \begin{pmatrix}\epsilon & 0\\ 0 & 1 \end{pmatrix} U_{\bar{1}} \begin{pmatrix} \epsilon^{-1} & 0\\ 0 & 1 \end{pmatrix}$$
			
			\item The same conclusion as in (2) and (3) holds for the $\SL_2(k)$-conjugacy classes of negative unipotent elements. 
		\end{enumerate}
	\end{proposition}
	
	\begin{proof}
		The proof follows directly from the previous lemma and \cite[2.7]{wa}. See also \cite[3.2]{mac}. 
	\end{proof}
	
	\noindent The following table summarizes the conjugacy classes in $\SL_2(k)$. The notations for the unipotent and negative unipotent classes in $\SL_2(k)$ given in the table will be used throughout this article.
	
	\begin{table}[ht]
		\begin{tabular}{ |c|c|c| }
			\hline
			Class type & Representative & Trace \\ 
			\hline
			Central & $\pm I$ & $\pm 2$\\
			\hline
			Unipotent $U_{\overline{\epsilon}};\; \overline{\epsilon} \in k^{\times}/k^{\times 2}$ & $\left(\begin{array}{cc} 1 & \epsilon \\ 0 & 1 \end{array}\right)$ & 2\\
			\hline 
			Negative unipotent $U_{-1,\overline{\epsilon}};\;\overline{\epsilon} \in k^{\times}/k^{\times 2}$ & $\begin{pmatrix} -1 & \epsilon \\ 0 & -1 \end{pmatrix}$ & $-2$\\
			\hline 
			Split semisimple & $\begin{pmatrix} \alpha & 0 \\ 0 & \alpha^{-1} \end{pmatrix}$; $\alpha\in k^{\times}\setminus \{\pm 1\}$ & $\alpha+\alpha^{-1}$\\
			\hline
			Non-split semisimple & $\begin{pmatrix} 0 & -1 \\ 1 & \lambda \end{pmatrix}$; $\lambda \in k$ such that $x^2-\lambda x+1$  & $\lambda$ \\
			&is irreducible over $k$ &\\
			\hline
		\end{tabular}
		\hspace{2 cm}
		\caption{\label{table1}Conjugacy classes in $\SL_2(k)$; $k$ is a field with $u(k)\leq 2$.}
	\end{table}
	
	\subsection{Reality properties in $\SL_2(k)$}
	
	Let $G$ be a group. An element $x\in G$ is called real if $x$ and $x^{-1}$ are conjugate in $G$. The conjugacy class of a real element is called a real conjugacy class. Let $\C$ be a conjugacy class in $G$. Let $\C^{-1}:=\{x^{-1}\mid x\in \C\}$. Thus $\C$ is a real conjugacy class if and only if $\C=\C^{-1}$. In this subsection, we determine real conjugacy classes in $\SL_2(k)$. The following lemma is obvious.
	
	\begin{lemma}\label{inverse_trace}
		Let $x\in \SL_2(k)$. Then $\tr(x)=\tr(x^{-1})$.
	\end{lemma}
	
	\noindent The following proposition characterizes real conjugacy classes in $\SL_2(k)$.
	
	\begin{proposition}\label{reality_properties_SL2}
		Let $k$ be a field with $u(k)\leq 2$ and let $\C$ be a conjugacy class in $\SL_2(k)$. The followings hold:
		\begin{enumerate}
			\item If $\C$ is semisimple, then $\C$ is a real conjugacy class.
			
			\item For $\overline{\epsilon}\in k^{\times}/k^{\times 2}$, we have: $U_{\overline{\epsilon}}^{-1}=U_{\overline{-\epsilon}}$ and $U_{-1,\overline{\epsilon}}^{-1}=U_{-1,\overline{-\epsilon}}.$
			
			\item For $\overline{\epsilon}\in k^{\times}/k^{\times 2}$, $U_{\overline{\epsilon}}$ is a real conjugacy class in $\SL_2(k)$ if and only if $-1\in k^{\times 2}$. The same holds for $U_{-1,\overline{\epsilon}}$.
		\end{enumerate}
	\end{proposition}
	
	\begin{proof}
		For (1), the proof follows from the previous lemma and \Cref{conjugacy_SL}(1). A direct computation yields (2). Finally,  (3) follows from (2) and \Cref{conjugacy_SL}(2).
	\end{proof}
	\subsection{Bruhat decomposition in $\SL_2(k)$}
	
	The Bruhat decomposition of a Chevalley group $G$ (see \cite[Proposition 8.3]{ca}) describes the double cosets of a Borel subgroup $B$ of $G$, in terms of the Weyl group $N_G(T)/T$ of a maximal torus $T\subseteq B$. We apply this to the group $\SL_2(k)$. 	Let $B:=\left\{\begin{pmatrix}
		a & b \\ 0 & a^{-1}
	\end{pmatrix} \mid a\in k^{\times}, b\in k\right\}$ be the subgroup of upper-triagular matrices in $\SL_2(k)$. This is a Borel subgroup of $\SL_2(k)$ containing the diagonal torus $H:=\left\{\begin{pmatrix} a & 0\\ 0 & a^{-1} \end{pmatrix} \mid a\in k^{\times} \right\}$. The normalizer of $H$  in $\SL_2(k)$ is the set of all monomial matrices in $\SL_2(k)$ and it is easy to see that the Weyl group is isomorphic to $S_2$, the symmetric group in two letters. Let $\n=\begin{pmatrix} 0 & 1 \\ -1 & 0 \end{pmatrix}$ be the non-trivial Weyl group element.  Let $\h(a)=\begin{pmatrix} a & 0 \\ 0 & a^{-1} \end{pmatrix}$ and define $\n(a):=\h(a)\n=\begin{pmatrix} 0 & a \\ -a^{-1} & 0 \end{pmatrix}$. Note that $\n=\n(1)$. With these notations, we have
	$$\SL_2(k)=B \sqcup B\n B.$$
	Let $\x_{12}(t)=\begin{pmatrix} 1 & t \\ 0 & 1 \end{pmatrix}$, where $t\in k$. Any element of $B$ can be written uniquely in the form  $\h(a)\x_{12}(t)$ and any element of $B\n B$ can be written uniquely in the form $\x_{12}(t)\n(a)\x_{12}(s)$, where $a\in k^{\times}$ and $t,s\in k$. More generally, this is called the \emph{``sharp form of Bruhat decomposition''} in case of arbitrary Chevalley groups (see  \cite[Theorem 8.4.3]{ca}). Thus every element in $\SL_2(k)$ can be written uniquely in any one of the above two forms. The following proposition gives the expressions for the product of two such elements.

	\begin{proposition}{\cite[Proposition 2.2]{ks}}
	\label{product_Bruhat_elements}
		With the notations as above
		\begin{enumerate}
			\item $\h(\alpha_1)\x_{12}(\psi_1). \h(\alpha_2)\x_{12}(\psi_2)=\h(\alpha_1\alpha_2)\x_{12}(\alpha_2^{-2}\psi_1+\psi_2),$
			\item $\h(\alpha_1)\x_{12}(\psi_1).\x_{12}(\tau_2)\n(\alpha_2)\x_{12}(\psi_2)=\x_{12}(\alpha_1^2(\psi_1+\tau_2))\n(\alpha_1\alpha_2)\x_{12}(\psi_2),$
			\item $\x_{12}(\tau_1)\n(\alpha_1)\x_{12}(\psi_1).\x_{12}(\tau_2)\n_(\alpha_2)\x_{12}(\psi_2)$\\
			=$\begin{cases}
				\x_{12}(\tau_1-\frac{\alpha_1^2}{\psi_1+\tau_2})\n(-\frac{\alpha_1\alpha_2}{\psi_1+\tau_2})\x_{12}(\psi_2-\frac{\alpha_2^2}{\psi_1+\tau_2}) & \text{when } \psi_1+\tau_2\neq 0\\
				\h(-\frac{\alpha_1}{\alpha_2})\x_{12}(\frac{\alpha_2^2}{\alpha_1^2}\tau_1+\psi_2) & \text{when } \psi_1+\tau_2= 0.
			\end{cases}$
		\end{enumerate}
	\end{proposition}
	
	\noindent The trace of the matrices written in the above two forms will be required.
	\begin{lemma}\label{trace_Bruhat_elements}
		The trace of $\h(a)\x_{12}(t)$ is $a+a^{-1}$ and that of $\x_{12}(t)\n(a)\x_{12}(s)$ is $-a^{-1}(t+s)$.
	\end{lemma}
	
	We end this section with some general observations on products of conjugacy classes in $\SL_2(k)$. The following result of Macbeath (originally over finite fields) carries over to arbitrary fields $k$  with $u(k)\leq 2$ as it is.
	
	\begin{theorem}{\cite[Theorem 1]{ma}}\label{trace_result}
		Let $k$ be a field with $u(k)\leq 2$ and $\mathsf{char}(k)\neq 2$. Let $\alpha, \beta, \gamma \in k$. Then there exists $A,B,C\in \SL_2(k)$ such that $\tr(A)=\alpha$, $\tr(B)=\beta$, and $\tr(C)=\gamma$ with $ABC=I$.
	\end{theorem}
	
	\noindent Due to \Cref{conjugacy_SL}(1), we get the following as an immediate corollary.
	
	\begin{corollary}\label{product_semisimple_classes_contain_semisimple}
	 Let $k$ be a field with $u(k)\leq 2$ and let $\C_1$ and $\C_2$ be two semisimple conjugacy classes in $\SL_2(k)$. Then $\C_1\C_2$ contains all semisimple elements of $\SL_2(k)$. 
	\end{corollary}
	
	\noindent The following two elementary lemmas will be useful.
	
	\begin{lemma}\label{prep_lemma_1}
		Let $G$ be a group and $H$ be a normal subgroup of $G$. Let $\C_1$ and $\C_2$ be two conjugacy classes in $H$ which are conjugacy classes in $G$ as well. Let $x\in H$ be such that  $x\in \C_1\C_2$. Then all $G$-conjugates of $x$ lie in $\C_1\C_2$.
	\end{lemma}

	\begin{proof} Let $x\in H$ be such that $x\in \C_1\C_2$. Then $x=c_1c_2$ for some elements $c_1\in \C_1$ and $c_2\in \C_2$. For any $g\in G$, we have 
		$$gxg^{-1}=g(c_1c_2)g^{-1}= (gc_1g^{-1})(gc_2g^{-1})\in \C_1\C_2$$ 
		since $\C_1, \C_2$ are also conjugacy classes in  $G$.
	\end{proof}
	
	\begin{lemma}\label{prep_lemma_2}
		Let $G$ be a group and $\C_1$, $\C_2$ be two conjugacy classes in $G$. Let $x\in \C_1$ and $y\in \C_2$. If $\sigma \in \C_1\C_2$, then $\sigma$ is conjugate to an element of the form $gxg^{-1}y$ for some $g\in G$.
	\end{lemma}
	
	\begin{proof}
		Let  $\sigma \in \C_1\C_2$. Then $\sigma=g_1xg_1^{-1}g_2yg_2^{-1}$ for some $g_1,g_2\in G$. We get $g_2^{-1}\sigma g_2=gxg^{-1}y$, where $g=g_2^{-1}g_1$.
	\end{proof}

	\section{Product of two conjugacy classes in $\SL_2(k)$ and $\psl_2(k)$}\label{product_of_two_classes}
	
	In this section, we give complete description of the product of two conjugacy classes in $\SL_2(k)$ for fields $k$ with $u(k)\leq 2$. We start with the product of two semisimple conjugacy classes.
	
	\subsection{Product of two semisimple conjugacy classes in $\SL_2(k)$ and $\psl_2(k)$}
	
	\begin{lemma}\label{lemma_1_sec_3}
		Let $k$ be a field with $u(k)\leq 2$ and let $\C_1$ and $\C_2$ be two semisimple conjugacy classes in $\SL_2(k)$. If $\C_1\C_2$ contains a unipotent element, then it contains all unipotent elements of $\SL_2(k)$. The same holds true for negative unipotent elements of $\SL_2(k)$.
	\end{lemma}
	
	\begin{proof}
		The lemma follows from  \Cref{conjugacy_SL} and \Cref{prep_lemma_1}.
	\end{proof}

	\noindent The following proposition completely describes the square of a semisimple conjugacy class in $\SL_2(k)$.
	\begin{proposition}\label{square_semisimple_class_SL}
	Let $k$ be a field with $u(k)\leq 2$ and let $G:=\SL_2(k)$. Let $\C$ be a semisimple conjugacy class in $G$. 
		\begin{enumerate}
			\item If $\C$ is split semisimple then $\C^2=G\setminus \{-I\}$ unless $\tr(\C)=0$, in which case $-1\in k^{\times 2}$ and $\C^2=G$.
			\item If $\C$ is non-split semisimple then $\C^2=G\setminus \{\{unipotents\}\cup\{-I\}\}$ unless $\tr(\C)=0$, in which case $-1 \notin k^{\times 2}$ and 
			$$\C^2=G\setminus \{\{unipotents\}\cup\{negative\; unipotents\}\}.$$
		\end{enumerate}
	\end{proposition}
	
	\begin{proof}
		Note that $\C^2$ contains all semisimple elements by \Cref{product_semisimple_classes_contain_semisimple}.
		
		Now we consider elements of trace 2. By \Cref{reality_properties_SL2}(1), $I\in \C^2$. Let $\C$ be the conjugacy class of $\begin{pmatrix}\alpha & 0 \\ 0 & \alpha^{-1}\end{pmatrix}$, where $\alpha \in k^{\times}$. We have
		$$\begin{pmatrix}\alpha & 0 \\ 0 & \alpha^{-1}\end{pmatrix}\begin{pmatrix}\alpha^{-1} & 1 \\ 0 & \alpha\end{pmatrix}=\begin{pmatrix} 1 & \alpha \\ 0 & 1\end{pmatrix}.$$
		By \Cref{lemma_1_sec_3}, we conclude that $\C^2$ contains all unipotent elements. Let $\C$ be a non-split semisimple class. Note that any element of $\C$ is contained in $B\n B$. We claim that $U:=\begin{pmatrix} 1 & 1 \\ 0 & 1 \end{pmatrix}\notin \C^2$. On the contrary, assume that there exists $x,y\in \C$ such that $xy=U$. Let $x=\x_{12}(\tau_1)\n(\alpha_1)\x_{12}(\psi_1)$ and $y=\x_{12}(\tau_2)\n(\alpha_2)\x_{12}(\psi_2)$. Note that since $xy\in B$, by \Cref{product_Bruhat_elements}, we must have $\psi_1+\tau_2=0$. Since $U=\h(1)\x_{12}(1)$, by \Cref{product_Bruhat_elements}, we have $\alpha_1=-\alpha_2$ and hence $\tau_1+\psi_2=1$. Thus we conclude that $\tau_1+\psi_1+\tau_2+\psi_2=1$. On the other hand as $x,y\in \C$, we have $\tr(x)=\tr(y)$, whence by \Cref{trace_Bruhat_elements}, we get that $\alpha_1^{-1}(\tau_1+\psi_1)=\alpha_2^{-1}(\tau_2+\psi_2)$. As $\alpha_1=-\alpha_2$, the last equality yields that  $\tau_1+\psi_1+\tau_2+\psi_2=0$, which is a contradiction. By \Cref{lemma_1_sec_3}, we conclude that $\C^2$ does not contain any unipotent element. 
		
		Finally we consider elements of trace $-2$. Note that if $-I\in \C^2$, then $\tr(\C)=0$. Suppose that $\tr(\C)\neq 0$. Then $-I\notin \C^2$. Since $\C^2$ contains a negative unipotent element by \Cref{trace_result}, using \Cref{lemma_1_sec_3} we conclude that $\C^2$ contains all negative unipotent elements. 
		
		Assume now that $\tr(\C)=0$. Then $$\begin{pmatrix}0 & -1\\ 1 & 0 \end{pmatrix}^2=-I.$$ Note that $-1\in k^{\times 2} \Leftrightarrow \C\;\text{is split semisimple}$. Assume that $-1\in k^{\times 2}$. Then $\C$ is the conjugacy class of $diag(\zeta,\zeta^{-1})$, where $\zeta \in k^{\times}$ and $\zeta^2=-1$. Note that 
		$$\begin{pmatrix}\zeta & 0 \\ 0 & \zeta^{-1}\end{pmatrix}\begin{pmatrix}\zeta & 1 \\ 0 & \zeta^{-1}\end{pmatrix}=\begin{pmatrix} -1 & \zeta \\ 0& -1\end{pmatrix}.$$
		By \Cref{lemma_1_sec_3}, $\C^2$ contains all negative unipotent elements. Now assume that  $-1\notin k^{\times 2}$. Then $\C$ is a non-split semisimple conjugacy class. We claim that $U':=\begin{pmatrix} -1 & 1 \\ 0 & -1 \end{pmatrix}\notin \C^2$. On the contrary, assume that there exists $x,y\in \C$ such that $xy=U'$. Let $x=\x_{12}(\tau_1)\n(\alpha_1)\x_{12}(\psi_1)$ and $y=\x_{12}(\tau_2)\n(\alpha_2)\x_{12}(\psi_2)$. Note that since $xy\in B$, by \Cref{product_Bruhat_elements}, we must have $\psi_1+\tau_2=0$. Since $U'=\h(-1)\x_{12}(-1)$, by \Cref{product_Bruhat_elements}, we have $\alpha_1=\alpha_2$ and hence $\tau_1+\psi_2=-1$. Thus we conclude that $\tau_1+\psi_1+\tau_2+\psi_2=-1$. On the other hand as $x,y\in \C$, we have $\tr(x)=\tr(y)=0$, whence by \Cref{trace_Bruhat_elements}, we get that $\alpha_1^{-1}(\tau_1+\psi_1)=\alpha_2^{-1}(\tau_2+\psi_2)=0$. Thus, $\tau_1+\psi_1+\tau_2+\psi_2=0$, which is a contradiction. By \Cref{lemma_1_sec_3}, $\C^2$ does not contain any negative unipotent element. The proof is now complete.
	\end{proof}
	
	\noindent The next proposition describes the product of two distinct semisimple conjugacy classes in $\SL_2(k)$.
	
	\begin{proposition}\label{product_distinct_semisimple_class_SL}
	Let $k$ be a field with $u(k)\leq 2$ and let $G:=\SL_2(k)$. Let $\C_1$ and $\C_2$ be two distinct semisimple conjugacy classes in $G$.
		\begin{enumerate}
			\item If $\tr(\C_1)\neq -\tr(\C_2)$, then $\C_1\C_2=G\setminus \{\pm I\}$.
			\item If $\tr(\C_1)=-\tr(\C_2)$, then $\C_1\C_2=G\setminus \{I\}$ unless $\C_1$ and $\C_2$ both are non-split semisimple, in which case, we have
			$$\C_1\C_2=G\setminus \{\{I\}\cup\{negative\;unipotents\}\}.$$  
		\end{enumerate}
	\end{proposition}
	
	\begin{proof}
		By \Cref{product_semisimple_classes_contain_semisimple}, $\C_1\C_2$ contains all semisimple elements of $\SL_2(k)$.
		
		Now we consider elements of trace 2. Note that by \Cref{reality_properties_SL2}(1), $I\notin \C_1\C_2$. By \Cref{trace_result}, there is an element in $\C_1\C_2$ whose trace is 2. Such an element must be unipotent, and therefore by \Cref{lemma_1_sec_3}, $\C_1\C_2$ contains all unipotent elements of $\SL_2(k)$ . 
		
		Finally we consider elements of trace $-2$. It is easy to see that $-I\in \C_1\C_2$ implies that $\tr(\C_1)=-\tr(\C_2)$. In other words, if $\tr(\C_1)\neq -\tr(\C_2)$, then $-I\notin \C_1\C_2$. Since $\C_1\C_2$ contains an element of trace $-2$ by \Cref{trace_result}, we conclude that such an element must be a negative unipotent element. Hence $\C_1\C_2$ contains all negative unipotent elements by \Cref{lemma_1_sec_3}. Now assume that $\alpha=\tr(\C_1)=-\tr(\C_2)$. Note that $\alpha\neq 0$ as $\C_1\neq \C_2$. We have
		$$\begin{pmatrix}\frac{\alpha}{2} & \frac{\alpha}{2}-\frac{2}{\alpha} \\ \frac{\alpha}{2} & \frac{\alpha}{2} \end{pmatrix}\begin{pmatrix}-\frac{\alpha}{2} & \frac{\alpha}{2}-\frac{2}{\alpha} \\ \frac{\alpha}{2} & -\frac{\alpha}{2} \end{pmatrix}=-I.$$
		If $\C_1$ is the conjugacy class of $diag(\beta,\beta^{-1})$, where $\beta \in k^{\times}$, then $\C_2$ is necessarily split semisimple and $\C_2$ is the class of $diag(-\beta, -\beta^{-1})$. We have
		$$\begin{pmatrix}-\beta & 0 \\ 0 & -\beta^{-1}\end{pmatrix}\begin{pmatrix}-\beta^{-1} & 1 \\ 0 & -\beta\end{pmatrix}=\begin{pmatrix} -1 & -\beta \\ 0 & -1\end{pmatrix}.$$
		By \Cref{lemma_1_sec_3}, $\C_1\C_2$ contains all negative unipotent elements. Now assume that $\C_1$ is non-split semisimple. Then $\C_2$ is necessarily non-split semisimple. Hence $\C_1,\C_2\subseteq B\n B$. We claim that $V:=\begin{pmatrix} -1 & 1 \\ 0 & -1 \end{pmatrix}\notin \C_1\C_2$. On the contrary, assume that there exists $x\in \C_1$ and $y\in \C_2$ such that $xy=V$. Let $x=\x_{12}(\tau_1)\n(\alpha_1)\x_{12}(\psi_1)$ and $y=\x_{12}(\tau_2)\n(\alpha_2)\x_{12}(\psi_2)$. Note that since $xy\in B$, by \Cref{product_Bruhat_elements}, we must have $\psi_1+\tau_2=0$. Since $V=\h(-1)\x_{12}(-1)$, by \Cref{product_Bruhat_elements}, we have $\alpha_1=\alpha_2$ and hence $\tau_1+\psi_2=-1$. Thus we conclude that $\tau_1+\psi_1+\tau_2+\psi_2=-1$. On the other hand as $x,y\in \C$, $\tr(x)=-\tr(y)$, whence by \Cref{trace_Bruhat_elements}, we get that $-\alpha_1^{-1}(\tau_1+\psi_1)=\alpha_2^{-1}(\tau_2+\psi_2)$. As $\alpha_1=\alpha_2$, the last equality yields that  $\tau_1+\psi_1+\tau_2+\psi_2=0$, which is a contradiction. By \Cref{lemma_1_sec_3}, we conclude that $\C_1\C_2$ does not contain any negative unipotent element of $\SL_2(k)$. The proof is now complete.
	\end{proof}
	
	\begin{proof}[\textbf{Proof of \Cref{square_semisimple_class_PSL}}] Let $\C$ be the conjugacy class of $\bar{x}$, where $x$ is a semisimple element of $\SL_2(k)$. Let $\D_1$ denote the conjugacy class of $x$ in $\SL_2(k)$ and $\D_2$ denote the conjugacy class of $-x$ in $\SL_2(k)$. Note that $\C^2$ contains every element $\bar{g}$, where $g\in \SL_2(k)$ and there exists $y\in \bar{g}$ such that $y$ is in one of $\D_1^2$, $\D_2^2$, or $\D_1\D_2$. The proof now follows from \Cref{square_semisimple_class_SL} and \Cref{product_distinct_semisimple_class_SL}. 
	\end{proof}

	\begin{proof}[\textbf{Proof of \Cref{product_distinct_semisimple_class_PSL}}]
		By considerations similar to the previous proof, the proof follows from  \Cref{product_distinct_semisimple_class_SL}.
	\end{proof}
	
	\subsection{Product of two  unipotent/negative unipotent conjugacy classes in $\SL_2(k)$ and $\psl_2(k)$} 
	
	We describe the square of a unipotent (resp. negative unipotent) conjugacy class in $\SL_2(k)$ and the product of two distinct unipotent (resp. negative unipotent) conjugacy classes in $\SL_2(k)$. Due to \Cref{prep_lemma_2}, we have the following observation.
	
	\begin{obs}
		\label{general_element_in_product_of_distinct_unipotent_classes_upto_conjugacy}
		Up to conjugacy, a general element of $U_{\overline{\epsilon_1}}U_{\overline{\epsilon_2}}$ is of the form $$M(a,b,c,d):= 
		\begin{pmatrix} a & b \\ c & d \end{pmatrix} 
		\begin{pmatrix}1 & \epsilon_1 \\ 0 & 1\end{pmatrix}
		\begin{pmatrix} a & b \\ c & d \end{pmatrix}^{-1} 
		\begin{pmatrix}1 & \epsilon_2 \\ 0 & 1\end{pmatrix}= 
		\begin{pmatrix}1-\epsilon_1 ac & \epsilon_2 +\epsilon_1 a^2 -\epsilon_1 \epsilon_2 ac \\ -\epsilon_1 c^2 & 1+ \epsilon_1 ac-\epsilon_1 \epsilon_2 c^2\end{pmatrix};$$
		where $a,b,c,d\in k$ with $ad-bc=1$. Note that the trace of $M(a,b,c,d)$ equals $2-\epsilon_1 \epsilon_2 c^2$.
	\end{obs}
	
	\begin{lemma}\label{semisimple_in_product of_distinct_unipotents}
		Let $k$ be a field and let $\overline{\epsilon_1}$, $\overline{\epsilon_2}\in k^{\times}/k^{\times 2}$. 
		Let $\C$ be a semisimple conjugacy class in $\SL_2(k)$. 
		Then $\C \subseteq U_{\overline{\epsilon_1}}U_{\overline{\epsilon_2}}$ if and only if $\overline{2- \tr(\C)}=\overline{\epsilon_1 \epsilon_2} \in k^{\times}/k^{\times 2}$.
	\end{lemma}
	
	\begin{proof}
		The proof follows from the expression of $\tr(M(a,b,c,d))$ computed in the previous observation.
	\end{proof}
	\begin{lemma}\label{unipotent_square_contains_itself}
		Let $k$ be a field with $|k|>5$. 
		Then $U_{\overline{\epsilon}}\subseteq U^2_{\overline{\epsilon}}$.
	\end{lemma}
	
	\begin{proof}
		By \Cref{conjugacy_SL}(3), it is enough to show that $U_{\bar{1}}\subseteq U_{\bar{1}}^2$. By \Cref{general_element_in_product_of_distinct_unipotent_classes_upto_conjugacy}, it is enough to find  $a,b,c,d\in k$ with $ad-bc=1$ such that $M(a,b,c,d)\in U_{\bar{1}}$. If $M(a,b,c,d)\in \SL_2(k)$ is such a matrix, then $\tr(M(a,b,c,d))=2$. By the previous observation, we have $2-c^2=2$ which implies that $c=0$. Note that $c=0$ implies $a\neq 0$. Thus 
		$M(a,b,c,d)=\begin{pmatrix}1 & 1 +a^2\\ 0 & 1\end{pmatrix}$. 
		Hence, $M(a,b,c,d)\in U_{\bar{1}}$ if and only if 
		$1+a^2\in k^{\times 2}$ for some $a\in k^{\times}$. Now since we have $|k|>5$, we can choose an element $\alpha \in k^{\times}$ such that $\alpha^2 \notin \{1,-1\}$. For such an $\alpha$, consider $a=\frac{\alpha-\alpha^{-1}}{2}$ and $b=\frac{\alpha+\alpha^{-1}}{2}$ . Then $a,b\neq 0$ and $b^2=1+a^2$. Hence we conclude that $U_{\bar{1}} \subseteq U_{\bar{1}}^2$.
	\end{proof}
	
	\begin{lemma}\label{unipotent_square_contains_other_unipotents}
		Let $k$ be a field. Assume that $u(k)\leq 2$ or $-1\in k^{\times 2}$. Then $U_{\overline{\epsilon_1}}\subseteq U^2_{\overline{\epsilon}}$ for all ${\overline{\epsilon_1}}, {\overline{\epsilon}} \in k^{\times}/k^{\times 2}$ 
		with ${\overline{\epsilon_1}}\neq {\overline{\epsilon}}$.
	\end{lemma}
	
	\begin{proof}
		By \Cref{conjugacy_SL}(3), it is enough to show $U_{\overline{\epsilon_1}}\subseteq U_{\bar{1}}^2$, where $\epsilon_1\notin k^{{\times}2}$. Since $\tr(U_{\overline{\epsilon_1}})=2$, arguing as in the proof of Lemma \ref{unipotent_square_contains_itself}, we get that $U_{\overline{\epsilon_1}}\subseteq U^2_1$ if and only if $\epsilon_1$ is a sum of two elements in $k^{\times 2}$. 
		
		Suppose first that $-1\in k^{\times 2}$. 
		Let $\iota$ denote an element in $k^{\times}$ such that $\iota^2=-1$. 
		Consider any $\alpha \in k^{\times}$. 
		Then for $a=\frac{\alpha+\epsilon_1\alpha^{-1}}{2}$ and $b=\frac{\epsilon_1\alpha^{-1}-\alpha}{2\iota}$, 
		we get that $a,b\neq 0$ since otherwise, $\epsilon_1= (\iota \alpha)^2$ or $\epsilon_1= \alpha^2$, which is a contradiction. Also, $\epsilon_1=a^2+b^2$. 
		
		Now suppose that $k$ is a field with $u(k)\leq 2$. By previous paragraph, we need to consider only the case when $-1\notin k^{\times 2}$. Consider the equation $x^2+y^2=\epsilon_1 z^2$. By our assumption on the field, there exists a non-trivial solution $(x_0,y_0,z_0)$ for the above equation. Since $-1\notin k^{\times 2}$, we have $z_0\neq0$. Since $\epsilon_1 \notin k^{\times 2}$, we get that $x_0,y_0\neq 0$. Thus we conclude that $\epsilon_1$ is a sum of two elements in $k^{\times 2}$.
	\end{proof}
	
	\begin{obs}\label{general_element_of_negative_unipotent_class}
		A general element of the conjugacy class $U_{-1,\overline{\epsilon}}$ is of the following form:\\
		$N_{\overline{\epsilon}}(e,f,g,h):=
		\begin{pmatrix} e & f \\ g & h\end{pmatrix} 
		\begin{pmatrix} -1 & \epsilon \\ 0 & -1\end{pmatrix}
		\begin{pmatrix} e & f \\ g & h\end{pmatrix}^{-1}= 
		\begin{pmatrix} -1-\epsilon eg & \epsilon e^2 \\ -\epsilon g^2 & -1+\epsilon eg\end{pmatrix}$; where $e,f,g,h$ vary over elements in $k$ with $eh-fg=1$.  
	\end{obs}
	
	\begin{lemma}\label{trace_negative_two_in_unipotent_square}
		Let $k$ be a field. Then the followings hold:
		\begin{enumerate}
			\item $-I \notin U^2_{\overline{\epsilon}}$,
			
			\item $U_{-1,\overline{\epsilon_1}} \nsubseteq U^2_{\overline{\epsilon}}$ for all $\overline{\epsilon}, \overline{\epsilon_1} \in k^{\times}/k^{\times 2}$ with ${\overline{\epsilon_1}}\neq {\overline{\epsilon}}$, and
			\item  $U_{-1,\overline{\epsilon}} \subseteq U^2_{\overline{\epsilon}}$ for all $\overline{\epsilon} \in k^{\times}/k^{\times 2}$.
		\end{enumerate}
	\end{lemma}
	
	\begin{proof}
		By \Cref{conjugacy_SL}(3), it is enough to prove the result for $U_{\bar{1}}^2$. Let $\C$ be a conjugacy class in $\SL_2(k)$ such that $\tr(\C)=-2$. By \Cref{general_element_in_product_of_distinct_unipotent_classes_upto_conjugacy},  $\C\subseteq U_{\bar{1}}^2$ if and only if there exists 
		$a,b,c,d\in k$ with $ad-bc=1$ such that $M(a,b,c,d)$ belongs to $\C$. Suppose we have such a matrix $M(a,b,c,d)\in \SL_2(k)$. Since we require $\tr(M(a,b,c,d))=-2$, we must have $2-c^2=-2$, which implies that 
		$c=\pm 2$. Thus we have two possibilities for the matrix $M(a,b,c,d)$:
		$$M_1:=\begin{pmatrix}1-2a & (a-1)^2 \\ -4 & 2a-3\end{pmatrix},\;\text{or}\; M_2:=\begin{pmatrix}1+2a & (1+a)^2 \\ -4 & -2a-3\end{pmatrix}.$$ 
		It is immediate that $-I \notin U_{\bar{1}}^2$. This proves $(1)$.
		
		By Observation \ref{general_element_of_negative_unipotent_class}, for 
		$\overline{\epsilon_1}\in k^{\times}/k^{\times 2}$, we  have that $U_{-1,\overline{\epsilon_1}}\subseteq U^2_{\bar{1}}$ if and only if 
		there exists $e,f,g,h \in k$ with $eh-fg=1$ such that $N_{\overline{\epsilon_1}}(e,f,g,h)=M_1$ or $N_{\overline{\epsilon_1}}(e,f,g,h)=M_2$. 
		Suppose that $\epsilon_1\notin k^{\times 2}$ and $U_{-1,\overline{\epsilon_1}}\subseteq U^2_{\bar{1}}$. 
		Then comparing the $(2,1)$ entries of the matrices $M_1$ and $M_2$ with $N_{\overline{\epsilon_1}}(e,f,g,h)$, we get that $\epsilon_1 g^2=4\in k^{\times 2}$, which is a contradiction. This proves $(2)$. 
		Since there exists at least one conjugacy class $\C$ in $U_{\bar{1}}^2$ with $\tr(\C)=-2$ (see \Cref{trace_result}), we conclude $U_{-1,\bar{1}}\subseteq U_{\bar{1}}^2$ and (3) follows.
	\end{proof}
	
	Combining \Cref{reality_properties_SL2}(2), \Cref{semisimple_in_product of_distinct_unipotents}, \Cref{unipotent_square_contains_itself}, \Cref{unipotent_square_contains_other_unipotents}, and \Cref{trace_negative_two_in_unipotent_square}, we can now give a description of the square of a unipotent conjugacy class in $\SL_2(k)$. 
	\begin{proposition}\label{unipotent_square_description}
		Let $k$ be a field with $u(k)\leq 2$ and $|k|>5$. Let $\C$ be a conjugacy class in $\SL_2(k)$. Then, for $\overline{\epsilon}\in k^{\times}/k^{\times 2}$, $\C\subseteq U_{\overline{\epsilon}}^2$ if and only if 
		\begin{enumerate}
			\item $\C$ is a semisimple conjugacy class with $2-\tr(\C)\in k^{\times 2}$.
			\item $\C$ is a unipotent conjugacy class.
			\item $\C=U_{-1,\overline{\epsilon}}$.
			\item $\C=\{I\}$ provided $-1\in k^{\times 2}$.
		\end{enumerate}
	\end{proposition}
	
	Now we describe the product of two distinct unipotent classes in $\SL_2(k)$.
	\begin{lemma}\label{field_lemma_1}
		Let $k$ be a field with $u(k)\leq 2$. Let $\epsilon, \epsilon_1, \epsilon_2 \in k^{\times}$ be such that $\overline{\epsilon}, \overline{\epsilon_1}, \overline{\epsilon_2}$ are distinct elements of $k^{\times}/k^{\times 2}$, and $\overline{\epsilon_2}\neq \overline{-\epsilon_1}$. Then there exists $a\in k^{\times}$ such that $\overline{\epsilon_2+\epsilon_1 a^2}=\overline{\epsilon} \in k^{\times}/k^{\times 2}$. 
	\end{lemma}
	
	\begin{proof}
		Consider the quadratic form $Q(x,y,z)=\epsilon_1 x^2+ \epsilon_2 y^2-\epsilon z^2$ over $k$. Since $u(k)\leq 2$, there exists $x_0,y_0,z_0\in k$, $(x_0,y_0,z_0)\neq(0,0,0)$ such that $Q(x_0,y_0,z_0)=0$. Since $\overline{\epsilon}, \overline{\epsilon_1}, \overline{\epsilon_2}$ are distinct classes in $k^{\times}/k^{\times 2}$ and $\overline{\epsilon_2}\neq \overline{-\epsilon_1}$, we conclude that $x_0,y_0,z_0\neq 0$. Thus $\frac{x_0}{y_0}\neq 0$, and $\overline{\epsilon_2+\epsilon_1 (\frac{x_0}{y_0})^2}=\overline{\epsilon}$. 
	\end{proof}
	
	\begin{lemma}\label{field_lemma_2}
		Let $k$ be a field. Then, for any $\epsilon \in k^{\times}$ such that $\overline{\epsilon}\neq \overline{\pm 1}$, there exists $b,c\in k^{\times}$ such that $\epsilon = b^2-c^2$. 
	\end{lemma}
	
	\begin{proof}
		Since $\overline{\epsilon} \neq \overline{\pm 1}$, we have $\epsilon \neq \pm 1$. 
		Let $b=\frac{1+\epsilon}{2}$ and $c=\frac{1-\epsilon}{2}$. Then 
		$b,c\neq 0$ and $b^2-c^2=\epsilon$. 
	\end{proof}

	\begin{lemma}\label{unipotents_in_product_of_distinct_unipotents}
		Let $k$ be a field with $u(k)\leq 2$. 
		Let $\overline{\epsilon_1}$, $\overline{\epsilon_2}$ be two distinct elements of  $k^{\times}/k^{\times 2}$. Then the product $U_{\overline{\epsilon_1}}U_{ \overline{\epsilon_2}}$ contains all unipotent elements of $\SL_2(k)$.
	\end{lemma}
	
	\begin{proof} 
		Assume that $\overline{\epsilon_2}=\overline{-\epsilon_1}$. In this case, $-1\notin k^{\times 2}$ since $\overline{\epsilon_1}\neq \overline{\epsilon_2}$. Using \Cref{reality_properties_SL2}(2), note that $U_{\overline{\epsilon_1}}\subseteq U_{\overline{\epsilon_1}} U_{\overline{-\epsilon_1}} \Leftrightarrow U_{\overline{\epsilon_1}}\subseteq U_{\overline{\epsilon_1}}^2$, where the latter holds by \Cref{unipotent_square_contains_itself}. Similarly, we get that $U_{\overline{-\epsilon_1}}\subseteq U_{\overline{\epsilon_1}} U_{\overline{-\epsilon_1}}.$ Consider $\overline{\epsilon_3}\neq \overline{\pm \epsilon_1}$. Then, $U_{\overline{\epsilon_3}}\subseteq U_{\overline{\epsilon_1}} U_{\overline{-\epsilon_1}}$ if and only if there exists $a,b,c,d\in k$ with $ad-bc=1$ and $M(a,b,c,d)\in U_{\overline{\epsilon_3}}$. By \Cref{general_element_in_product_of_distinct_unipotent_classes_upto_conjugacy}, if such $M(a,b,c,d)$ exists, then $\tr(M(a,b,c,d))=2+\epsilon_1^2c^2=2$. This implies $c=0$. Thus we get $M(a,b,c,d)=\begin{pmatrix}1 & \epsilon_1 (a^2-1)\\ 0 & 1\end{pmatrix}$. Then, $M(a,b,c,d)\in U_{\overline{\epsilon_3}}$ if and only if there exists $b,c\in k^{\times}$ such that $\epsilon_1^{-1}\epsilon_3=b^2-c^2$. This holds by \Cref{field_lemma_2} as $\overline{\epsilon_1^{-1}\epsilon_3}\neq \overline{\pm 1}$. We conclude that all unipotent elements of $\SL_2(k)$ are contained in $U_{\overline{\epsilon_1}}U_{\overline{-\epsilon_1}}$.
		
		Let us now assume that $\overline{\epsilon_2}\neq \overline{-\epsilon_1}$. Let $\overline{\epsilon}\neq \overline{\epsilon_1},\overline{\epsilon_2}$. Using \Cref{general_element_in_product_of_distinct_unipotent_classes_upto_conjugacy}, and arguing as in the previous paragraph, we get that $U_{\overline{\epsilon}}\subseteq U_{\overline{\epsilon_1}}U_{\overline{\epsilon_2}}$ if and only if there exists $a\in k^{\times}$ such that $\overline{\epsilon_2+\epsilon_1 a^2}=\overline{\epsilon}.$ This is true by \Cref{field_lemma_1}, whence it follows that $U_{\overline{\epsilon}}\subseteq U_{\overline{\epsilon_1}}U_{\overline{\epsilon_2}}$. Further, $U_{\overline{\epsilon_1}}\subseteq U_{\overline{\epsilon_1}}U_{\overline{\epsilon_2}}\Leftrightarrow U_{\overline{\epsilon_2}}\subseteq U_{\overline{\epsilon_1}}U_{\overline{-\epsilon_1}}$. The latter holds true by \Cref{unipotent_square_contains_other_unipotents} when $-1\in k^{\times 2}$ and by the previous paragraph when $-1\notin k^{\times 2}$. Similarly we conclude that $U_{\overline{\epsilon_2}}\subseteq U_{\overline{\epsilon_1}}U_{\overline{\epsilon_2}}$ and our proof is complete.
	\end{proof}
	
	\begin{lemma}\label{trace_minus_two_in_product of_distinct_unipotents}
		Let $k$ be a field. 
		Let $\overline{\epsilon_1}$, $\overline{\epsilon_2}$ be two distinct elements of  $k^{\times}/k^{\times 2}$. 
		Then $U_{\overline{\epsilon_1}}U_{\overline{\epsilon_2}}$ does not contain any element of trace $-2$. 
	\end{lemma}
	
	\begin{proof}
		Let $\C$ be a conjugacy class in $\SL_2(k)$ with $\tr(\C)=-2$. 
		Using \Cref{general_element_in_product_of_distinct_unipotent_classes_upto_conjugacy},
		$\C \subseteq U_{\overline{\epsilon_1}}U_{\overline{\epsilon_2}}$ 
		if and only if there exists $a,b,c,d$ with $ad-bc=1$ such that 
		$M(a,b,c,d)\in \C$. Suppose we have such a matrix $M(a,b,c,d)\in \SL_2(k)$. 
		Then, by \Cref{general_element_in_product_of_distinct_unipotent_classes_upto_conjugacy}, $\tr(M(a,b,c,d))=2-\epsilon_1\epsilon_2 c^2=-2$ $\Rightarrow \epsilon_1\epsilon_2c^2 = 4 \in k^{\times 2}$; hence we get a contradiction. 
	\end{proof}
	
	Combining \Cref{reality_properties_SL2}(2), \Cref{semisimple_in_product of_distinct_unipotents}, \Cref{unipotents_in_product_of_distinct_unipotents}, and \Cref{trace_minus_two_in_product of_distinct_unipotents}, we can now give a description of the product of two distinct unipotent classes in $\SL_2(k)$.
	
	\begin{proposition}\label{product_of_two_distinct_unipotents}
		Let $k$ be a field with $u(k)\leq 2$. Let $\C$ be a conjugacy class in $\SL_2(k)$. Then, for $\overline{\epsilon_1}, \overline{\epsilon_2} \in k^{\times}/k^{\times 2}$ with $\overline{\epsilon_1}\neq \overline{\epsilon_2}$, $\C\subseteq U_{\overline{\epsilon_1}}U_{\overline{\epsilon_2}}$ if and only if 
		\begin{enumerate}
			\item $\C$ is a semisimple conjugacy class with $2-\tr(\C)\in \epsilon_1\epsilon_2k^{\times 2}$.
			\item $\C$ is a unipotent conjugacy class.
			\item $\C=\{I\}$ provided $-1\notin k^{\times 2}$ and $\overline{\epsilon_2}=\overline{-\epsilon_1}$.
		\end{enumerate}
	\end{proposition}
	
	Now we deal with squares and products of conjugacy classes of negative unipotent elements of $\SL_2(k)$. We merely state the final results for this, since the proofs are very similar to the unipotent case.
	
	\begin{proposition}\label{negative_unipotent_square_description}
		Let $k$ be a field with $u(k)\leq 2$ and $|k|>5$. Let $\C$ be a conjugacy class in $\SL_2(k)$. Then, for $\overline{\epsilon}\in k^{\times}/k^{\times 2}$, $\C\subseteq U_{-1,\overline{\epsilon}}^2$ if and only if 
		\begin{enumerate}
			\item $\C$ is a semisimple conjugacy class with $2-\tr(\C)\in k^{\times 2}$.
			\item $\C$ is a unipotent conjugacy class.
			\item $\C=U_{-1,\overline{-\epsilon}}$.
			\item $\C=\{I\}$ provided $-1\in k^{\times 2}$.
		\end{enumerate}
	\end{proposition}
	
	\begin{proposition}\label{product_of_two_distinct_negative_unipotents}
		Let $k$ be a field with $u(k)\leq 2$. Let $\C$ be a conjugacy class in $\SL_2(k)$. Then, for $\overline{\epsilon_1}, \overline{\epsilon_2} \in k^{\times}/k^{\times 2}$ with $\overline{\epsilon_1}\neq \overline{\epsilon_2}$, $\C\subseteq U_{-1,\overline{\epsilon_1}}U_{-1,\overline{\epsilon_2}}$ if and only if 
		\begin{enumerate}
			\item $\C$ is a semisimple conjugacy class with $2-\tr(\C)\in \epsilon_1\epsilon_2k^{\times 2}$.
			\item $\C$ is a unipotent conjugacy class.
			\item $\C=\{I\}$ provided $-1\notin k^{\times 2}$ and $\overline{\epsilon_2}=\overline{-\epsilon_1}$.
		\end{enumerate}
	\end{proposition}
	
	\noindent Now we describe  the product of a unipotent class and a negative unipotent class in $\SL_2(k)$.
	\begin{proposition}\label{product_of_unipotent_negative_unipotent_class_SL}
		Let $k$ be a field with $u(k)\leq 2$. Let $\C$ be a conjugacy class in $SL_2(k)$. Then, for any $\overline{\epsilon_1}, \overline{\epsilon_2}\in k^{\times}/k^{\times 2}$,  $\C\subseteq U_{\overline{\epsilon_1}}U_{-1,\overline{\epsilon_2}}$ if and only if 
		\begin{enumerate}
			\item $\C$ is a semisimple conjugacy class such that $-2-\tr(\C)\in \epsilon_1\epsilon_2k^{\times 2}.$
			
			\item $\C=U_{\overline{-\epsilon_1}}$,  provided $\overline{\epsilon_2}=\overline{-\epsilon_1}.$
			
			\item $\C$ is a negative unipotent conjugacy class.

			\item $\C=\{-I\}$, provided $\overline{\epsilon_1}=\overline{\epsilon_2}$.
		\end{enumerate}
	\end{proposition}
	
	\begin{proof}
		Most of the above theorem will directly follow from previous results. 
Using \Cref{prep_lemma_2}, any element of $U_{\overline{\epsilon_1}}U_{-1,\overline{\epsilon_2}}$ is conjugate to a matrix of the following form:
		$$P(a,b,c,d)= 
		\begin{pmatrix} a & b \\ c & d \end{pmatrix} 
		\begin{pmatrix}1 & \epsilon_1 \\ 0 & 1\end{pmatrix}
		\begin{pmatrix} a & b \\ c & d \end{pmatrix}^{-1} 
		\begin{pmatrix}-1 & \epsilon_2 \\ 0 & -1\end{pmatrix}= 
		\begin{pmatrix}-1+\epsilon_1 ac & \epsilon_2 -\epsilon_1 a^2 -\epsilon_1 \epsilon_2 ac \\ \epsilon_1 c^2 & -1-\epsilon_1 ac-\epsilon_1 \epsilon_2 c^2\end{pmatrix};$$
		where $a,b,c,d\in k$ with $ad-bc=1$. We observe that the trace of $P(a,b,c,d)$ equals $-2-\epsilon_1 \epsilon_2 c^2$. Clearly, $I\notin U_{\overline{\epsilon_1}}U_{-1,\overline{\epsilon_2}}$.
		Let $\C$ be a semisimple class.Then $\C\subseteq U_{\overline{\epsilon_1}}U_{-1,\overline{\epsilon_2}}$ if and only if $-2-\tr(\C)\in \epsilon_1\epsilon_2k^{\times 2}$. 
		Let $\overline{\epsilon_3}\in k^{\times}/k^{\times 2}$. For unipotent classes, we have $U_{\overline{\epsilon_3}}\subseteq U_{\overline{\epsilon_1}}U_{-1,\overline{\epsilon_2}}\Leftrightarrow U_{-1,\overline{\epsilon_2}}\subseteq U_{\overline{\epsilon_3}}U_{\overline{-\epsilon_1}}$. By \Cref{trace_negative_two_in_unipotent_square}(3) and \Cref{product_of_two_distinct_unipotents}, the later holds true if and only if $\overline{-\epsilon_1}=\overline{\epsilon_3}=\overline{\epsilon_2}$. This proves that for a unipotent class $\C$, $\C\subseteq U_{\overline{\epsilon_1}}U_{-1,\overline{\epsilon_2}}$ if and only if $\C=U_{\overline{-\epsilon_1}}$,  provided $\overline{\epsilon_2}=\overline{-\epsilon_1}$. For negative unipotent classes, we have $U_{-1,\overline{\epsilon_3}}\subseteq U_{\overline{\epsilon_1}}U_{-1,\overline{\epsilon_2}}\Leftrightarrow U_{\overline{\epsilon_1}}\subseteq U_{-1,\overline{\epsilon_3}}U_{-1,\overline{-\epsilon_2}}$, where the latter holds true from the previous two propositions. We conclude that $U_{\overline{\epsilon_1}}U_{-1,\overline{\epsilon_2}}$ contains all negative unipotent elements of $\SL_2(k)$. Finally, by taking $c=0$ in the matrix $P(a,b,c,d)$ and observing that (1,2) position of $P(a,b,c,d)$ is $0$ if and only if $\overline{\epsilon_1}=\overline{\epsilon_2}$, we conclude that $-I \subseteq U_{\overline{\epsilon_1}}U_{-1,\overline{\epsilon_2}}$ if and only if $\overline{\epsilon_1}=\overline{\epsilon_2}$.
	\end{proof}
	
	\begin{proof}[\textbf{Proof of \Cref{product_of_distinct_unipotent_class_PSL}}] Let $\C_1$ and $\C_2$ be as in the statement of the theorem. Then, $\bar{x}\in \C_1\C_2$, (where $x\in \SL_2(k)$) if and only if there exists $y\in \SL_2(k)$ such that $\bar{y}=\bar{x}$ and $y$ is in one of $U_{\overline{\epsilon_1}}U_{\overline{\epsilon_2}}$, $U_{-1,\overline{-\epsilon_1}}U_{-1,\overline{-\epsilon_2}}$, $U_{\overline{\epsilon_1}}U_{-1,\overline{-\epsilon_2}}$, or $U_{-1,\overline{-\epsilon_1}}U_{\overline{\epsilon_2}}$. When $|k|>5$, the proof follows directly from \Cref{unipotent_square_description}, \Cref{product_of_two_distinct_unipotents}, \Cref{negative_unipotent_square_description}, \Cref{product_of_two_distinct_negative_unipotents} and \Cref{product_of_unipotent_negative_unipotent_class_SL}. When $|k|=5$, for $\bar{\epsilon}\in k^{\times}/k^{\times 2}$, using \Cref{trace_negative_two_in_unipotent_square} we get $U_{-1,\overline{\epsilon}}\subseteq  U_{\bar{\epsilon}}^2$. Note that the only problem when $|k|=5$ comes from \Cref{unipotent_square_contains_itself} which becomes redundant in $\psl_2(k)$ as $\overline{\begin{pmatrix}1 & \epsilon \\ 0 & 1\end{pmatrix}}=\overline{\begin{pmatrix}-1 & \epsilon \\ 0 & -1\end{pmatrix}}$ in $\psl_2(k)$ in this case. Thus our proof follows in this case as well.
	\end{proof}
	
	\medskip
	
	\noindent \cite[Theorem 3]{ga} describes the square of a unipotent class in $\psl_2(q)$ in terms of orders of its elements. For completeness, we state \Cref{product_of_distinct_unipotent_class_PSL} for the finite field $\F_q$ in terms of order of the elements. We require the notion of elements of $q$-good order in $\psl_2(q)$ (see \cite[Definition 2]{ga}). We recall it here once again but we separate out the two types of semisimple elements (that is, split and non-split) in the original definition, as it will help us to state our result with clarity.
	
	\begin{definition}[Elements of $q$-good order] Let $q>3$ be odd. Suppose $x\in \SL_2(q)$ be semisimple and $\bar{x}$ be its image in $\psl_2(q)$. Suppose  $|\bar{x}|=t$. 
		\begin{enumerate}
			\item If $x$ is split semisimple, we say $\bar{x}$ is of $q$-good order if $t$ is odd and $t\mid q-1$, or $t$ is even and $4t\mid q-1$.
			
			\item If $x$ is non-split semisimple, we say $\bar{x}$ is of $q$-good order if $t$ is odd and $t\mid q+1$, or $t$ is even and $4t\mid q+1$.
		\end{enumerate}
	\end{definition}
	
	For an element $x\in \SL_2(q)$, we say that $\bar{x}\in \psl_2(q)$ is of \textbf{$q$-bad order} if it is not of $q$-good order. We have the following result.
	
	\begin{theorem}\label{product_distinct_unipotent_finite_fields}
		Let $G:=\psl_2(q)$ and $\C_1,\C_2$ be the two distinct unipotent conjugacy classes in $\psl_2(q)$.
		\begin{enumerate}
			\item If $q\equiv 1(\rm{mod}\;4)$ then,
			$$\C_1\C_2=\{unipotents \} \cup \{non\text{-}split\;semisimple\;elements\}\cup \{split \;semisimple\;elements\;of\;q\text{-}bad\;order\}.$$
			
			\item If $q\equiv 3(\rm{mod}\;4)$ then,
			$$\C_1\C_2=\{unipotents \} \cup \{split\;semisimple\;elements\}\cup \{non\text{-}split \;semisimple\;elements\;of\;q\text{-}bad\;order\} \cup\{1\}.$$
		\end{enumerate}
	\end{theorem}
	
	\begin{proof}
	 By \Cref{product_of_distinct_unipotent_class_PSL}, $\C_1\C_2$ contains all unipotent elements. Further, by \Cref{reality_properties_SL2}, $1\in \C_1\C_2$ if and only if $q\equiv 3(\text{mod }4)$. Now we deal with semisimple elements of $\psl_2(q)$. We make use of \cite[Proposition 9]{ga}, which states that if $\bar{x}\in \psl_2(q)$ (where $x\in \SL_2(q)$ is semisimple), then $\bar{x}$ is of $q$-good order if and only if one of $2-\tr(x)$ or $2+\tr(x)$ is a square in $\F_q^{\times}$. Thus, $\bar{x}$ is of $q$-bad order if and only if both $2+\tr(x)$ and $2-\tr(x)$ are non-squares in $\F_q^{\times}$.
	 
	 Suppose $x\in \SL_2(q)$ is semisimple. By \Cref{product_of_distinct_unipotent_class_PSL}, $\bar{x}\in \C_1\C_2$ if and only if one of $2-\tr(x)$ or $2+\tr(x)$ is a non-square in $\F_q^{\times}$. We make two separate cases.
	 
	 \medskip
	 
	 \noindent \textbf{Case I:} Suppose $x$ is split semisimple. Then $\tr(x)=a+a^{-1}$ for some $a\in \F_q^{\times}$. Note that $2+\tr(x)=\frac{(a+1)^2}{a}$ and $2-\tr(x)=-\frac{(a-1)^2}{a}$. If $q\equiv 1(\text{mod }4)$, then either both $2-\tr(x)$ and $2+\tr(x)$ are squares (when $a\in \F_q^{\times 2}$), or both are non-squares (when $a\notin \F_q^{\times 2}$). Thus using the first paragraph, we conclude that $\bar{x}\in \C_1\C_2$ if and only if  $\bar{x}$ is of $q$-bad order. Suppose now that $q\equiv 3(\text{mod }4)$. In this case, note that one of $2+\tr(x)$ or $2-\tr(x)$ is always a non-square and hence we conclude that $\C_1\C_2$ contains all split semisimple elements.
	 
	 \medskip
	 
	 \noindent \textbf{Case II:} Suppose $x$ is non-split semisimple. Then $\tr(x)=a+a^q$, where $a\in \F_{q^2}^1\setminus\{\pm 1\}$. Here $\F_{q^2}^{1}:=\{y \in \F_{q^2}\mid y^{1+q}=1\}$. Since $a^{q}=a^{-1}$ in $\F_{q^2}$, we get that $2+\tr(x)=\frac{(a+1)^2}{a}$ and $2-\tr(x)=-\frac{(a-1)^2}{a}$. We claim that $\frac{(a+1)^2(a-1)^2}{a^2}=(a-a^{-1})^2\notin \F_q^{\times 2}$. On the contrary, let $x\in \F_q^{\times}$ be such that $x^2=(a-a^{-1})^2$. Then $x=\pm(a-a^{-1})\in \F_q^{\times}$. Since $a+a^{-1}\in \F_q$, we conclude that $a\in \F_q$, a contradiction to our assumption. Assume that $q\equiv 1(\text{mod }4)$. We claim that one of $2-\tr(x)$ or $2+\tr(x)$ is a non-square in $\F_q^{\times}$. On the contrary, assume that both $2-\tr(x)$ and $2+\tr(x)$ is a square, which implies that $(2+\tr(x))(2-\tr(x))\in \F_{q}^{\times 2}$. Thus, we conclude that  $\frac{(a+1)^2(a-1)^2}{a^2}\in \F_q^{\times 2}$ (as $-1\in \F_q^{\times 2})$, a contradiction to our previous claim. This proves the claim, and by the first paragraph of the proof, we conclude that $\C_1\C_2$ contains all non-split semisimple elements in this case. Assume $q\equiv 3(\text{mod }4)$. We claim that either both $2-\tr(x)$ and $2+\tr(x)$ are squares, or both are non-squares. On the contrary, let us assume that one of them is a square and the other is a non-square. Then $-\frac{(a+1)^2(a-1)^2}{a^2}\notin \F_q^{\times 2}\Rightarrow \frac{(a+1)^2(a-1)^2}{a^2}\in \F_q^{\times 2}$ (as $-1\notin \F_q^{\times 2})$. This contradicts one of our previous claim and we conclude $\bar{x}\in \C_1\C_2$ if and only if $\bar{x}$ is of $q$-bad order (see the first paragraph of the proof). Our proof is now complete.
	\end{proof}
	We now proceed towards the final subsection of this section.
	
	\subsection{Product of a unipotent/negative unipotent  and a semisimple class in $\SL_2(k)$ and $\psl_2(k)$} 
	
	\begin{proposition}\label{semisimple_times_unipotent}
		Let $k$ be a field with $u(k)\leq 2$. Let $\overline{\epsilon}\in k^{\times}/k^{\times 2}$, and $\C$ be a semisimple conjugacy class in $\SL_2(k)$. Then the followings hold:
		\begin{enumerate}
			\item $\C U_{\overline{\epsilon}}$ does not contain $\pm I$. 
			\item $\C U_{\overline{\epsilon}}$ contains $\C$ if and only if $\C$ is split semisimple. 
			\item $\C U_{\overline{\epsilon}}$ contains $\C'$ for all semisimple conjugacy class $\C'$ with $\C'\neq \C$. 
			\item For a unipotent conjugacy class $U_{\overline{\epsilon_1}}$,  $\C U_{\overline{\epsilon}}$ contains $U_{\overline{\epsilon_1}}$ if and only if $\overline{\tr(\C)-2} =\overline{\epsilon \epsilon_1}\in k^{\times}/k^{\times 2}$. 
			\item For a negative unipotent conjugacy class $U_{-1,\overline{\epsilon_1}}$, $\C U_{\overline{\epsilon}}$ contains $U_{-1,\overline{\epsilon_1}}$ if and only if $\overline{\tr(\C)+2} =\overline{\epsilon \epsilon_1}\in k^{\times}/k^{\times 2}$. 
		\end{enumerate}
	\end{proposition}
	
	\begin{proof}
		Let $x\in \C$ and $y\in U_{\overline{\epsilon}}$ be such that $xy=\pm I$. Then $\tr(x)=\pm \mathsf{tr(y)}$. But then $y$ is semisimple, which is a contradiction. This proves $(1)$. Note that $\C U_{\overline{\epsilon}} \supseteq \C \Leftrightarrow U_{\overline{\epsilon}} \subseteq \C^2$. By \Cref{square_semisimple_class_SL}, the latter statement holds true if and only if $\C$ is split semisimple. This proves $(2)$. 
		Note that $\C U_{\overline{\epsilon}} \supseteq \C' \Leftrightarrow U_{\overline{\epsilon}} \subseteq \C\C'$; by \Cref{product_distinct_semisimple_class_SL}, the latter statement holds true. This proves $(3)$. Note that $\C U_{\overline{\epsilon}}\supseteq U_{\overline{\epsilon_1}} \Leftrightarrow \C\subseteq U_{\overline{-\epsilon}}U_{\overline{\epsilon_1}}$ (see \Cref{reality_properties_SL2}(2)). Using \Cref{semisimple_in_product of_distinct_unipotents}, we conclude that $\overline{\tr(\C)-2}=\overline{\epsilon\epsilon_1}$ in $k/k^{\times 2}$. This proves $(4)$. Again, $\C U_{\overline{\epsilon}} \supseteq U_{-1,\overline{\epsilon_1}} \Leftrightarrow \C \subseteq U_{\overline{-\epsilon}} U_{-1,\overline{\epsilon_1}}$; by \Cref{product_of_unipotent_negative_unipotent_class_SL}, the latter statement holds true if and only if $\overline{\tr(\C)+2}=\overline{\epsilon \epsilon_1} \in k^{\times}/k^{\times 2}$. This proves $(5)$. 
	\end{proof}
	
	We omit the proof of the next proposition since it can be obtained from the previous results, using similar arguments as in the proof of the previous proposition. 
	\begin{proposition}\label{semisimple_times_negative_unipotent}
		Let $k$ be a field with $u(k)\leq 2$. Let $\overline{\epsilon}\in k^{\times}/k^{\times 2}$ and $\C$ be a semisimple conjugacy class in $\SL_2(k)$. Then the followings hold:
		\begin{enumerate}
			\item $\C U_{-1,\overline{\epsilon}}$ does not contain $\pm I$. 
			\item $\C U_{-1,\overline{\epsilon}}$ contains $\C$ if and only if $\tr(\C)\neq 0$ or $-1\in k^{\times 2}$.
			\item For a semisimple conjugacy class $\C'$ with $\C'\neq \C$, $\C U_{-1,\overline{\epsilon}}$ contains $\C'$ 
			if and only if $\tr(\C)\neq-\tr(\C')$, or at least one of $\C$ or $\C'$ is split semisimple.
			\item For a unipotent conjugacy class $U_{\overline{\epsilon_1}}$, $\C U_{-1,\overline{\epsilon}}$ contains $U_{\overline{\epsilon_1}}$ if and only if $\overline{\tr(\C)+2} =\overline{\epsilon \epsilon_1}\in k^{\times}/k^{\times 2}$. 
			\item For a negative unipotent conjugacy class $U_{-1,\overline{\epsilon_1}}$, $\C U_{-1,\overline{\epsilon}}$  contains $U_{-1,\overline{\epsilon_1}}$ if and only if $\overline{\tr(\C)-2} =\overline{\epsilon \epsilon_1}\in k^{\times}/k^{\times 2}$. 
		\end{enumerate}
	\end{proposition}
	
	\begin{proof}[\textbf{Proof of \Cref{semisimple_times_unipotent_PSL}}]
		Arguing as in the proof of \Cref{square_semisimple_class_PSL} and \Cref{product_of_distinct_unipotent_class_PSL}, the proof follows  from the above two propositions.
	\end{proof}
	
	\section{Product of three distinct conjugacy classes in $\SL_2(k)$ and $\psl_2(k)$}\label{product_of_three_classes}
	
	In this section we prove \Cref{product_of_three_classes_PSL}. We have the following.
	
	\begin{proposition}\label{product_of_three_proposition}
		Let $k$ be a field with $u(k)\leq 2$. Let $G:=\SL_2(k)$ and let $\C_1,\C_2,\C_3$ be three conjugacy classes in $G$ with at least two of them distinct. Then $G\setminus \{\pm I\}\subseteq \C_1\C_2\C_3$.
	\end{proposition}
	
	\begin{proof}
		We first consider the case when $\C_1,\C_2, \C_3$ are all semisimple classes. Assume that $\C_1 \neq \C_2$. Then by \Cref{product_distinct_semisimple_class_SL}, $\C_1 \C_2$ contains all unipotent and semisimple elements. Let $\C$ be a split semisimple class contained in $\C_1 \C_2$. Then by \Cref{product_distinct_semisimple_class_SL}, $G\setminus \{\pm I \} \subseteq \C \C_3 \subseteq \C_1\C_2\C_3$.
		
		\medskip
		
		Now let us consider the case when $\C_1,\C_2$ are semisimple and $\C_3$ is unipotent. By using \Cref{semisimple_times_unipotent}(3), we conclude that $\C_2\C_3$ contains all semisimple classes $\C$ with $\C \neq \C_2$, and by \Cref{semisimple_times_unipotent}(2), we have $\C_2\C_3$ contains $\C_2$ if and only if $\C_2$ is split semisimple. First assume that $\C_1 = \C_2$ and they are split semisimple. Then ${\C_2}^2 \subseteq \C_1\C_2\C_3$, and therefore by \Cref{square_semisimple_class_SL}, $G\setminus \{-I\} \subseteq \C_1\C_2\C_3$. Now assume that $\C_1 = \C_2$ and they are non-split semisimple. By \Cref{semisimple_times_unipotent}(3), $\C_2\C_3$ contains all split semisimple classes. Let $\C$ be such a class. By \Cref{product_distinct_semisimple_class_SL}, we have $G\setminus \{\pm I \} \subseteq \C_1 \C \subseteq \C_1\C_2\C_3$. Now we assume that $\C_1 \neq \C_2$. By \Cref{product_distinct_semisimple_class_SL} and \Cref{unipotents_in_product_of_distinct_unipotents}, $\C_1\C_2\C_3$ contains all unipotent elements. Also, by using \Cref{semisimple_times_unipotent}(3), we have $\C_1 \subseteq \C_2\C_3$, whence ${\C_1}^2 \subseteq \C_1\C_2\C_3$. Hence, by \Cref{square_semisimple_class_SL}, we have $G\setminus \{\text{negative unipotent elements}\} \subseteq \C_1\C_2\C_3$ if $\tr(\C_1)=0$ and $\C_1$ is non-split semisimple, and $G\setminus \{-I\} \subseteq \C_1\C_2\C_3$ otherwise. Now, if $\tr(\C_1)=0$ then $\tr(\C_2) \neq 0$, whence by \Cref{product_distinct_semisimple_class_SL}, we conclude that $\C_1\C_2$ contains all negative unipotent elements. Thus, using \Cref{product_of_unipotent_negative_unipotent_class_SL}(3), $\C_1\C_2\C_3$ contains all negative unipotent elements. Hence, $G\setminus \{-I\} \subseteq \C_1\C_2\C_3$. The case when $\C_1,\C_2$ are semisimple classes and $\C_3$ is a negative unipotent class follows similarly.
		
		\medskip
		
		Let us now assume that only one of $\C_1,\C_2,\C_3$ is semisimple and let that class be $\C_1$. We also assume that $\C_2,\C_3$ are unipotent classes. Let $\C_2 =U_{\overline{\epsilon_1}}$ and $\C_3 =U_{\overline{\epsilon_2}}$ for some $\overline{\epsilon_1}, \overline{\epsilon_2} \in k^{\times}/k^{\times 2}$. Using \Cref{semisimple_in_product of_distinct_unipotents}, we conclude that if $\C$ is a semisimple class then $\C \subseteq \C_2\C_3$ if and only if $2-\tr(\C)\in \epsilon_1\epsilon_2k^{\times 2}$. Since $|k|>5$, there are at least three squares in $k^{\times}$, and hence there is at least one semisimple class $\C'$ with $\C' \neq \C_1$ and $\C'\subseteq \C_2\C_3$. If either $\C_1$ or $\C'$ is split semisimple then by \Cref{product_distinct_semisimple_class_SL}, $G\setminus \{\pm I\} \subseteq \C_1\C' \subseteq \C_1\C_2\C_3$. Now assume that $\C_1$ is non-split semisimple and $\C_2\C_3$ contains no split semisimple classes. If $|k|>7$, then there are at least three distinct non-split semisimple classes in $\C_2\C_3$. Thus, by \Cref{product_distinct_semisimple_class_SL}, $G\setminus \{\pm I\} \subseteq \C_1\C' \subseteq \C_1\C_2\C_3$. Also if $|k|=7$ and $\overline{\epsilon_1} \neq \overline{\epsilon_2}$, then there are three distinct choices for non-split semisimple classes $\C$ such that $2-\tr(\C)\in \epsilon_1\epsilon_2k^{\times 2}$. Hence, again by \Cref{product_distinct_semisimple_class_SL}, $G\setminus \{\pm I\} \subseteq \C_1\C' \subseteq \C_1\C_2\C_3$. Finally if $|k|=7$ and $\overline{\epsilon_1} = \overline{\epsilon_2}$, then the choices for trace of non-split semisimple classes $\C$ such that $2-\tr(\C)\in \epsilon_1\epsilon_2k^{\times 2}$ are $0$ and $1$. Thus, there exists a non-split semisimple class $\C \subseteq \C_2\C_3$ such that $\tr(\C) \neq \pm \tr(\C_1)$. By \Cref{product_distinct_semisimple_class_SL}, $G\setminus \{\pm I\} \subseteq \C_1\C' \subseteq \C_1\C_2\C_3$. The case when $\C_2,\C_3$ are negative unipotent classes, or $\C_2$ is a unipotent class and $\C_3$ is  a negative unipotent class follows similarly. 
		
		\medskip
		
		Finally, let us assume that none of $\C_1,\C_2$, and $\C_3$ is a semisimple class. We consider the case when all three classes are unipotent. The other cases follow similarly. Let $\C_1 =U_{\overline{\epsilon_1}}$, $\C_2 =U_{\overline{\epsilon_2}}$ and $\C_3 =U_{\overline{\epsilon_3}}$ for some $\overline{\epsilon_1}, \overline{\epsilon_2}$, and $ \overline{\epsilon_3} \in k^{\times}/k^{\times 2}$. We can also assume that $\overline{\epsilon_1} \neq \overline{\epsilon_2}$. Using \Cref{product_of_two_distinct_unipotents}(2), we conclude that $U_{\overline{\epsilon_1}} U_{\overline{\epsilon_2}}$ contains all unipotent elements, whence by \Cref{product_of_two_distinct_unipotents}(2) once again, $U_{\overline{\epsilon_1}} U_{\overline{\epsilon_2}} U_{\overline{\epsilon_3}}$ contains all unipotent elements. Let $\C$ be a semisimple class. By \Cref{semisimple_in_product of_distinct_unipotents}, $\C \subseteq \C_1\C_2$ if and only if $2-\tr(\C)\in \epsilon_1\epsilon_2k^{\times 2}$. If $|k|\geq 5$, there are at least two such classes. Let $\D_1$ and $\D_2$ be two semisimple classes such that $\D_1,\D_2\subseteq \C_1\C_2$. By \Cref{semisimple_times_unipotent}(3), note that $\D_1\C_3$  (resp. $\D_2\C_3$) contains all semisimple classes except possibly $\D_1$ (resp. $\D_2$). We conclude that  $\C_1\C_2\C_3$ contains all semisimple elements. Finally we show that every negative unipotent element is contained in $\C_1\C_2\C_3$. We first assume that $\overline{\epsilon_2} = \overline{\epsilon_3}$. In this case, by \Cref{trace_negative_two_in_unipotent_square}(3) we have $U_{-1,\overline{\epsilon_3}} \subseteq \C_2\C_3$ and hence by \Cref{product_of_unipotent_negative_unipotent_class_SL}(3), $\C_1\C_2\C_3$ contains all negative unipotent elements. Now let us assume that $\overline{\epsilon_1}, \overline{\epsilon_2}$, and $\overline{\epsilon_3}$ are all distinct. Let $\C$ be a semisimple class such that $2-\tr(\C)\in \epsilon_1\epsilon_2k^{\times 2}$ or, in other words, $\tr(\C) = 2- \epsilon_1\epsilon_2 a^2$ for some $a \in k^{\times}$. By \Cref{product_of_two_distinct_unipotents}(1), $\C\subseteq \C_1\C_2$. By \Cref{semisimple_times_unipotent}(5), $\C U_{\overline{\epsilon_3}}$ contains a negative unipotent class $U_{-1,\overline{\epsilon}}$ if and only if $2 + \tr(\C) = 4-\epsilon_1\epsilon_2 a^2 \in \epsilon \epsilon_3 k^{\times 2}$. Equivalently, $U_{-1,\overline{\epsilon}}\subseteq \C U_{\overline{\epsilon_3}}$ if and only if  $1-\epsilon_1\epsilon_2 a^2 = \epsilon \epsilon_3 b^2$ for some $b \in k^{\times}$. Consider the quadratic form $Q(x,y,z)=x^2-\epsilon_1\epsilon_2 y^2-\epsilon \epsilon_3 z^2$. Since $u(k)\leq 2$, we conclude that there exists $a, b \in k^{\times}$ such that $1-\epsilon_1\epsilon_2 a^2 = \epsilon \epsilon_3 b^2$, except for the case when $\overline{\epsilon} = \overline{\epsilon_3}$ or $\overline{\epsilon} = \overline{-\epsilon_1\epsilon_2\epsilon_3}$. Thus, if $\overline{\epsilon} \neq \overline{\epsilon_3}$ and $\overline{\epsilon} \neq \overline{-\epsilon_1\epsilon_2\epsilon_3}$ then $U_{-1,\overline{\epsilon}} \subseteq \C \C_3 \subseteq \C_1\C_2\C_3$. For $\overline{\epsilon} = \overline{\epsilon_3}$, by \Cref{trace_negative_two_in_unipotent_square}(3), $U_{-1,\overline{\epsilon_3}} \subseteq U_{\overline{\epsilon_3}}^2 \subseteq \C_1\C_2\C_3$. Finally, if $\epsilon = -\epsilon_1\epsilon_2\epsilon_3 c^2$ for some $c \in k^{\times}$, then we can assume that $\overline{\epsilon} \neq \overline{\epsilon_3}$ and hence $\overline{\epsilon_1} \neq \overline{-\epsilon_2}$. Thus, for $a=\frac{{\epsilon_1}^{-1}+{\epsilon_2}^{-1}}{2} \in k^{\times}$ and $b=\frac{{\epsilon_1}^{-1} -{\epsilon_2}^{-1}}{2\epsilon_3 c} \in k^{\times}$, we have $1-\epsilon_1\epsilon_2 a^2 = \epsilon \epsilon_3 b^2$. Hence $\C_1\C_2\C_3$ contains all negative unipotent elements and thus $G\setminus \{\pm I\} \subseteq \C_1\C_2\C_3$.
	\end{proof}

	\begin{proposition}
	Let $k=\F_5$. For $i=1,2,3$, let $x_i$ be a non-central element of $\SL_2(k)$. 
	Let $\C_i$ be the conjugacy class of $\overline{x_i}$ in $\psl_2(k)$ such that at least two of them are distinct. Then $$\C_1 \C_2 \C_3 \supseteq \psl_2(k)\setminus \{1\}.$$   
	\end{proposition}
	
\begin{proof}

Let $\D_i$ be the conjugacy class of $x_i$ in $\SL_2(k)$. 
Suppose that either none of $\D_i$ is semisimple, or at least two of them are semisimple. 
Then proceeding as in the proof of \Cref{product_of_three_proposition}, we conclude that $\D_1 \D_2 \D_3 \supseteq \SL_2(k)\setminus \{\pm I\}$. 
Hence $\C_1 \C_2 \C_3 \supseteq \psl_2(k)\setminus \{1\}$.

Now suppose that exactly one of $\D_i$ (say $\D_1$) is semisimple. 
We can assume that $\D_2$ and $\D_3$ are unipotent conjugacy classes. 
Let $\D_2=U_{\overline{\epsilon_2}}$ and $\D_3=U_{\overline{\epsilon_3}}$ for some $\overline{\epsilon_2}$, $\overline{\epsilon_3}\in k^{\times}/k^{\times 2}$. 
We first consider the case when $\mathsf{tr}(\D_1)=0$. In this case, $\D_1$ is split semisimple. 
Let $\D$ be a semisimple class. Then by \Cref{semisimple_in_product of_distinct_unipotents}, $\D\subseteq U_{\overline{\epsilon_2}} U_{\overline{\epsilon_3}}$ if and only if $2-\mathsf{tr}(\D)\in \epsilon_2 \epsilon_3 k^{\times 2}$. 
If $\overline{\epsilon_2}=\overline{\epsilon_3}$ (resp. $\overline{\epsilon_2}\neq \overline{\epsilon_3}$) then the non-split conjugacy class $\D$ of trace $1$ (resp. $-1$) is contained in $U_{\overline{\epsilon_2}} U_{\overline{\epsilon_3}}$. 
Thus, by \Cref{product_distinct_semisimple_class_SL}(1), $\SL_2(k)\setminus \{\pm I\}=\D_1 \D \subseteq \D_1 \D_2 \D_3$. 
Therefore, $\psl_2(k)\setminus \{1\} \subseteq \C_1 \C_2 \C_3$. 
Now suppose that $\mathsf{tr}(\D_1)=\pm 1$. 
It is enough to consider the case when $\mathsf{tr}(\D_1)=1$. 
Here $\D_1$ is non-split semisimple. 
If $\overline{\epsilon_2}=\overline{\epsilon_3}$ then by \Cref{semisimple_in_product of_distinct_unipotents}, $\D_1 \subseteq U_{\overline{\epsilon_2}} U_{\overline{\epsilon_3}}.$ 
Thus, by \Cref{square_semisimple_class_SL}(2), $\SL_2(k)\setminus \{\{-I\}\cup \{ unipotents \} \} = \D^2_1 \subseteq \D_1 \D_2 \D_3$. 
If $\overline{\epsilon_2}\neq \overline{\epsilon_3}$ then by \Cref{semisimple_in_product of_distinct_unipotents}, $-\D_1 \subseteq U_{\overline{\epsilon_2}} U_{\overline{\epsilon_3}}$, where 
$-\D_1$ is the conjugacy class of $-x_1$ in $\SL_2(k)$. 
Thus by \Cref{product_distinct_semisimple_class_SL}(2), $\SL_2(k) \setminus \{\{I\}\cup \{negative \; unipotents \} \}= \D_1 (-\D_1) \subseteq \D_1 \D_2 \D_3$. 
Hence, in this case, $\psl_2(k) =\C_1 \C_2 \C_3$.

\end{proof}

	\begin{proof}[\bf{Proof of \Cref{product_of_three_classes_PSL}}]
		The proof follows from the above two propositions.
	\end{proof}
	\section{A remark on commutators in $\SL_2(k)$ and $\psl_2(k)$}\label{commutator_map}
	
	The Ore conjecture, which states that every element in a finite simple group is a commutator is now a theorem (see \cite{lost}). For a finite group $G$ and $g\in G$, we say $g$ is a coprime commutator if $g=[x,y]$ for some $x,y \in G$ with $(|x|,|y|)=1$. Let $G$ be a finite simple group of Lie type over a finite field $\F_q$, where $q=p^{a}$ for some prime $p$ and $a\geq 1$. Then an element $g\in G$ is unipotent (resp. semisimple) if and only if $|g|=p^b$ for some $b\geq 1$ (resp. $(p,|g|)=1$). Shumyatsky in \cite{sh} proved that every element of the alternating group $A_n$ ($n\geq 5)$ is a coprime commutator and conjectured the same for any finite simple group. In \cite{ps}, Shumyatsky and Pellegrini showed that every element in $\psl_2(q)$ ($q\geq 4$) is a coprime commutator. In fact  for $\psl_2(q)$, they proved that every element is a commutator of an involution (that is, an element of order two) and an odd order element (see proofs of \cite[Theorem 1.1 \& Theorem 1.2]{ps}). 
	
	\medskip
	
	When $q$ is even, note that any involution is a unipotent element and any odd order element is semisimple, hence any element in $\psl_2(q)$ can be written as a commutator of a unipotent element and a semisimple element. It is natural to ask whether the same holds for $\psl_2(q)$ when $q$ is odd, and more generally for any finite simple group of Lie type. Unfortunately, the answer to this question is negative, although the next theorem shows that this holds for $\psl_2(k)$, when $k$ is quadratically closed.
	
	\begin{proposition}\label{semisimple_unipotent_commutator_SL_quad_closed}
		Let $k$ be a quadratically closed field. Then every element of $\SL_2(k)\setminus \{-I\}$ can be written as a commutator of a unipotent element and a semisimple element.
	\end{proposition}
	
	\begin{proof}
		Let $G:=\SL_2(k)$. Note that there is a single conjugacy class of unipotent elements  in $G$. Call it $U$. Consider $g\in U^2\setminus \{I\}$ such that $g=uxu^{-1}x^{-1}$, where $u=\begin{pmatrix}1 & 1 \\ 0 & 1\end{pmatrix}$ and $x :=\begin{pmatrix}a & b \\ c& d \end{pmatrix}\in G\setminus \Z_{G}(u)$. Note that 
		$$\Z_{G}(u)=\Z_{G}(u^{-1})=\left\{\begin{pmatrix}\alpha & \beta \\ 0 & \alpha^{-1}\end{pmatrix}\mid \alpha\in k^{\times}, \beta \in k\right\}.$$
		If $x$ is semisimple we are done. If not then $\tr(x)=\pm 2$. We have $a+d=\pm 2$ and since $x\notin \Z_{G}(u^{-1})$, we get that $c\neq 0$. We have
		$$x\Z_{G}(u^{-1})=\left\{\begin{pmatrix} a\alpha & a\beta+b\alpha^{-1} \\ c\alpha & c\beta+d\alpha^{-1} \end{pmatrix}\mid \alpha\in k^{\times}, \beta\in k\right\}.$$
		The trace of any element in $x\Z_{G}(u^{-1})$ is $a\alpha+c\beta+d\alpha^{-1}$. By an appropriate choice of $\alpha,\beta\in k$ (we can take $\alpha=1$ and $\beta\neq 0,\pm 4{c}^{-1}$), we see that there exists an element $y \in x\Z_{G}(u^{-1})$ such that $\tr(y)\neq \pm 2$, whence $y$ is semisimple. We conclude that $g=[u,y]$. If $g$ is a commutator of a unipotent element and a semisimple element then any conjugate of $g$ is so, and hence using \Cref{prep_lemma_2}, we conclude that every non-identity element of $U^2$ is a commutator of a unipotent element and a semisimple element. By \cite[Corollary 5.2]{vw}, $U^2=G\setminus \{-I\}$ and hence the proof follows for all elements of $G\setminus\{-I\}$ except for $I$. Finally, due to Jordan decomposition, there always exists a semisimple and a unipotent element that commutes, whence the result holds for $I$ as well.
	\end{proof}
	
	\begin{theorem}\label{semisimple_unipotent_commutator_PSL_quad_closed}
		Let $k$ be a quadratically closed field. Then every element of $\psl_2(k)$ can be written as a commutator of a unipotent element and a semisimple element.
	\end{theorem}

	\begin{proof}
		The proof follows from the previous proposition.
	\end{proof}

	\begin{remark}\label{concluding_remark}
		Rimhak Ree (see \cite{re}) proved that every element of a connected semisimple linear algebraic group defined over an algebraically closed field is a commutator. In \Cref{semisimple_unipotent_commutator_SL_quad_closed}, when $k$ is algebraically closed field, then $-I$ cannot be written as a commutator of a unipotent element and a semisimple element. More generally, if $G$ is a connected semisimple linear algebraic group defined over an algebraically closed field and $x$ is a nontrivial central element of $G$, then $x$ cannot be written as a commutator of a unipotent element and a semisimple element. On the contrary, assume that there exists a unipotent element $u\in G$ and a semisimple element $s\in G$ such that $x=sus^{-1}u^{-1}$. Consider $z=xu=sus^{-1}$. Thus $z$ is unipotent. But since $x$ is central and hence semisimple, $z=xu$ is the Jordan decomposition of $z$ with a non-trivial semisimple part, a contradiction. It is unknown to us whether every non-central element of a semisimple linear algebraic group over an algebraically closed field can be written as a commutator of a unipotent element and a semisimple element.
	\end{remark}
	
	As mentioned earlier, the following proposition shows that not every element of $\psl_2(q)$ is a commutator of  a semisimple element and a unipotent element.
	\begin{proposition}
		Let $q>3$ be odd and let $G:=\psl_2(q)$. Then an element $g\in G$ cannot be written as a commutator of a semisimple element and a unipotent element if and only if 
		\begin{enumerate}
			\item $g$ is a semisimple element of $q$-bad order when $q\equiv 1\;(\rm{mod}\;4)$, or 
			
			\item $g$ is a non-split semisimple element of $q$-good order when $q\equiv 3\;(\rm{mod}\;4)$.
		\end{enumerate}
	\end{proposition}
	
	\begin{proof}
		Let $q\equiv 1\;(\text{mod }4)$. Suppose $g\in G$ and $g=[x,y]=xyx^{-1}y^{-1}$, where $x\in G$ is semisimple and $y\in G$ is unipotent. Then $g\in U^2$, where $U$ is any one of the two unipotent conjugacy classes in $G$. By \cite[Theorem 3]{ga}, this implies that $g$ is identity, unipotent, or a semisimple element of $q$-good order. If $q\equiv 3\;(\text{mod }4)$, then $g\in \C_1\C_2$, where $\C_1$ and $\C_2$ are the two distinct unipotent classes in $\psl_2(q)$. Note that $\C_2=\C_1^{-1}$. By \Cref{product_distinct_unipotent_finite_fields}, we conclude that $g$ is identity, unipotent, split semisimple, or a non-split semisimple element of $q$-bad order.
		
		The converse can be proved by similar considerations as in the proof of \Cref{semisimple_unipotent_commutator_SL_quad_closed}.
	\end{proof}

	\subsection*{Acknowledgment} The first named author acknowledges the support through Prime Minister's Research Fellowship from the Ministry of Education, Government of India (PMRF ID: 0601097). The second and third named authors would like to acknowledge the support of IISER Mohali institute post-doctoral fellowship during this work. We thank Prof. Maneesh Thakur for the argument given in \Cref{concluding_remark}. We also thank Prof. Amit Kulshrestha for his interest in this work.
	\bibliographystyle{abbrv}
	\bibliography{references}
\end{document}